\documentclass[12pt]{article}
\title{Ensemble estimators for multivariate \\entropy estimation}
\pdfoutput=1
\author{\textbf{Kumar Sricharan}, \textbf{Dennis Wei}, \textbf{Alfred O. Hero III}
\\ \{\textbf{kksreddy,dlwei,hero}\}@umich.edu \\
 Department of EECS, University of Michigan Ann Arbor
}

\date{\today}

\usepackage[OT1]{fontenc}
\usepackage{amsthm,amsmath}
\usepackage[numbers]{natbib}

\newcommand{\expect}{{{\mathbb{E}}}}
\newcommand{\mb}{\mathbf}
\newcommand{\var}{{{\mathbb{V}}}}

\usepackage[hmargin=2.5cm,vmargin=3cm]{geometry}
\usepackage{latexsym,amssymb,amsmath}
\usepackage{amsmath, amsthm, amssymb}
\usepackage{graphicx}
\usepackage[mathcal]{euscript}
\usepackage{algorithm}
\usepackage{algorithmic}
\usepackage{appendix}
\usepackage{subfigure}
\usepackage{mathrsfs}

\numberwithin{equation}{section}
\theoremstyle{plain}

\newtheorem{theorem}{Theorem}
\newtheorem{lemma}[theorem]{Lemma}

\begin{document}
\maketitle

\begin{abstract}
The problem of estimation of density functionals like entropy and mutual information has received much attention in the statistics and  information theory communities. A large class of estimators of functionals of the probability density suffer from the curse of dimensionality, wherein the mean squared error (MSE) decays increasingly slowly as a function of the sample size $T$ as the dimension $d$ of the samples increases. In particular, the rate is often  glacially slow of order  $O(T^{-{\gamma}/{d}})$, where $\gamma>0$ is a rate parameter. Examples of such estimators include kernel density estimators, $k$-nearest neighbor ($k$-NN)  density estimators, $k$-NN entropy estimators, intrinsic dimension estimators and other examples. In this paper, we propose a weighted affine combination of an ensemble of such estimators, where optimal weights can be chosen such that the weighted estimator converges at a much faster dimension invariant rate of $O(T^{-1})$. Furthermore, we show that these optimal weights can be determined by solving a convex optimization problem which can be performed offline and does not require training data. We illustrate the superior performance of our weighted estimator for two important applications: (i) estimating the Panter-Dite distortion-rate factor and (ii) estimating the Shannon entropy for testing the probability distribution of a random sample.
\end{abstract}

\section{Introduction}

Non-linear functionals of probability densities $f$ of the form $G(f) = \int g(f(x),x) f(x) dx$ arise in applications of information theory, machine learning, signal processing and statistical estimation. Important examples of such functionals include Shannon $g(f,x)=-\log(f)$ and R\'enyi $g(f,x) = f^{\alpha-1}$ entropy, and the quadratic functional $g(f,x) = f^{2}$. In these applications, the functional of interest often must be estimated empirically from sample realizations of the underlying densities. 

Functional estimation has received significant attention in the mathematical statistics community. However, estimators of functionals of multivariate probability densities $f$ suffer from mean square error (MSE) rates which typically decrease with dimension $d$ of the sample as $O(T^{-{\gamma}/{d}})$, where $T$ is the number of samples and $\gamma$ is a positive rate parameter. Examples of such estimators include kernel density estimators~\cite{kde}, $k$-nearest neighbor ($k$-NN) density estimators~\cite{fuk2}, $k$-NN entropy functional estimators~\cite{hero,kks,litt}, intrinsic dimension estimators~\cite{kks}, divergence estimators~\cite{wang}, and mutual information estimators. This slow convergence is due to the curse of dimensionality. In this paper, we introduce a simple affine combination of an ensemble of such slowly convergent estimators and show that the weights in this combination can be chosen to significantly improve the rate of MSE convergence of the weighted estimator.  In fact our ensemble averaging method can improve MSE convergence to the parametric rate $O(T^{-1})$.

Specifically, for $d$-dimensional data, it has been observed that the variance of estimators of functionals $G(f)$ decays as $O(T^{-1})$ while the bias decays as $O(T^{-1/(1+d)})$. To accelerate the slow rate of convergence of the bias in high dimensions, we propose a weighted ensemble estimator for ensembles of estimators that satisfy conditions ${\mathscr C}.1$(\ref{BE}) and ${\mathscr C}.2$(\ref{VE}) defined in Sec. II below. Optimal weights, which serve to lower the bias of the ensemble estimator to $O(T^{-1/2})$, can be determined by solving a convex optimization problem. Remarkably, this optimization problem does not involve any density-dependent parameters and can therefore be performed offline. This then ensures MSE convergence of the weighted estimator at the parametric rate of $O(T^{-1})$. 

\subsection{Related work}

When the density $f$ is $s > d/4$ times differentiable, certain estimators of functionals of the form $\int g(f(x),x) f(x) dx$, proposed by Birge and Massart~\cite{birge}, Laurent~\cite{laurent} and  Gin\'e and Mason~\cite{gini}, can achieve the parametric MSE convergence rate of $O(T^{-1})$. The key ideas in ~\cite{birge,laurent,gini} are: (i) estimation of quadratic functionals $\int f^2(x) dx$ with MSE convergence rate $O(T^{-1})$; (ii) use of kernel density estimators with kernels that satisfy the following symmetry constraints: 
\begin{equation}
\int K(x) dx =1, \hspace{0.5in} \int x^r K(x) dx = 0, 
\label{eq:symmetrickernels}
\end{equation} for $r=1,..,s$; and finally (iii) truncating the kernel density estimate so that it is bounded away from $0$. By using these ideas, the estimators proposed by ~\cite{birge,laurent,gini} are able to achieve parametric convergence rates.

In contrast, the estimators proposed in this paper require additional higher order smoothness conditions on the density, i.~e.~ the density must be $s>d$ times differentiable. However, our estimators are much simpler to implement in contrast to the estimators proposed in ~\cite{birge,laurent,gini}. In particular, the estimators in ~\cite{birge,laurent,gini} require separately estimating quadratic functionals of the form $\int f^2(x) dx$, and using truncated kernel density estimators with symmetric kernels (\ref{eq:symmetrickernels}), conditions that are not required in this paper. Our estimator is a simple affine combination of an ensemble of estimators, where the ensemble satisfies conditions ${\mathscr C}.1$ and ${\mathscr C}.2$. Such an ensemble can be trivial to implement. For instance, in this paper we show that simple uniform kernel plug-in estimators (\ref{eq:plugin}) satisfy conditions ${\mathscr C}.1$ and ${\mathscr C}.2$. 

Ensemble based methods have been previously proposed in the context of classification. For example, in both boosting~\cite{boosting} and multiple kernel learning~\cite{multipleK} algorithms, lower complexity weak learners are combined to produce classifiers with higher accuracy. Our work differs from these methods in several ways. First and foremost, our proposed method performs estimation rather than classification. An important consequence of this is that the weights we use are {\emph {data independent}}, while the weights in boosting and multiple kernel learning must be estimated from training data since they depend on the unknown distribution. 

\subsection{Organization}
The remainder of the paper is organized as follows. We formally describe the weighted ensemble estimator for a general ensemble of estimators in Section~\ref{sec:genmeth}, and specify conditions ${\mathscr C}.1$ and ${\mathscr C}.2$ on the ensemble that ensure that the ensemble estimator has a faster rate of MSE convergence. Under the assumption that conditions ${\mathscr C}.1$ and ${\mathscr C}.2$ are satisfied, we provide an MSE optimal set of weights as the solution to a  convex optimization(\ref{convexsoltheory}). Next, we shift the focus to entropy estimation in Section~\ref{sec:entropyest}, propose an ensemble of simple uniform kernel plug-in entropy estimators, and show that this ensemble satisfies conditions ${\mathscr C}.1$ and ${\mathscr C}.2$. Subsequently, we apply the ensemble estimator theory in Section~\ref{sec:genmeth} to the problem of entropy estimation using this ensemble of kernel plug-in estimators. We present simulation results in Section~\ref{sec:exp} that illustrate the superior performance of this ensemble entropy estimator in the context of (i) estimation of the Panter-Dite distortion-rate factor~\cite{riten} and (ii) testing the probability distribution of a random sample. We conclude the paper in Section~\ref{sec:con}.

\subsection*{Notation}
We will use bold face type to indicate random variables and random vectors and regular type face for constants. We denote the statistical expectation operator by the symbol $\expect$ and the conditional expectation given random variable $\mb{Z}$ using the notation $\expect_{\mb{Z}}$. We also define the variance operator as $\var[\mb{X}] = \expect[(\mb{X}-\expect[\mb{X}])^2] $ and the covariance operator as $\mathrm{Cov}[\mb{X},\mb{Y}] = \expect[(\mb{X}-\expect[\mb{X}])(\mb{Y}-\expect[\mb{Y}])] $. We denote the bias of an estimator by $\mathbb{B}$.

\section{Ensemble estimators}
\label{sec:problem}
\label{sec:genmeth}

Let $\bar{l} = \{l_1,..,l_L\}$ denote a set of parameter values. For a parameterized ensemble of estimators $\{\hat{\mb{E}}_l\}_{l \in \bar{l}}$ of $E$, define the weighted ensemble estimator with respect to weights $w = \{w(l_1),\ldots,w(l_L)\}$ as $$\hat{\mb{E}}_w = \sum_{l \in \bar{l}} w(l)\hat{\mb{E}}_l$$ where the weights satisfy $\sum_{l \in \bar{l}} w(l) = 1$. This latter sum-to-one condition guarantees that $\hat{\mb{E}}_w$ is asymptotically unbiased if the component estimators $\{\hat{\mb{E}}_l\}_{l \in \bar{l}}$ are asymptotically unbiased. Let this ensemble of estimators $\{\hat{\mb{E}}_l\}_{l \in \bar{l}}$ satisfy the following two conditions:
\begin{itemize}
 \item ${\mathscr C}.1$ The bias is given by 
\begin{eqnarray}
\label{Bias_Ensemble}
\label{BE}
\mathbb{B}({\hat{\mb{E}}}_{l}) &=& \sum_{i \in {\cal I}} c_{i}\psi_i(l)T^{-{i/2d}} + O(1/\sqrt{T}),
\end{eqnarray}
where $c_{i}$ are constants that depend on the underlying density, ${\cal I} = \{i_1,..,i_I\}$ is a finite index set with cardinality $I<L$, $\min({\cal I}) = i_0 > 0$ and $\max({\cal I}) = i_d \leq d$, and $\psi_i(l)$ are basis functions that depend only on the estimator parameter $l$.

\item ${\mathscr C}.2$ The variance is given by 
\begin{eqnarray}
\label{Variance_Ensemble}
\label{VE}
\mathbb{V}({\hat{\mb{E}}}_{l}) &=& c_{v}\left({\frac{1}{T}}\right) + o\left(\frac{1}{T}\right).
\end{eqnarray}

\end{itemize}

\begin{theorem}
\label{lemma:weightedensemble}
For an ensemble of estimators $\{\hat{\mb{E}}_l\}_{l \in \bar{l}}$, assume that the conditions ${\mathscr C}.1$ and ${\mathscr C}.2$ hold. Then, there exists a weight vector $w_o$ such that
$$\expect[(\hat{\mb{E}}_{w_o} - E)^2]  = O(1/T).$$ This weight vector can be found by solving the following convex optimization problem:
\begin{equation}
\begin{aligned}
& \underset{w}{\text{minimize}}
& & ||w||_2 \\
& \text{subject to}
& & \sum_{l \in \bar{l}} w(l) = 1, \\
&&& \gamma_w(i) = \sum_{l \in \bar{l}} w(l)\psi_i(l) = 0, \; i \in {\cal I},
\end{aligned}
\label{convexsoltheory}
\end{equation}
where $\psi_i(l)$ is the basis defined in (2.1).
\end{theorem}

\begin{proof}

The bias of the ensemble estimator is given by
\begin{eqnarray}
\label{Bias_EnsembleEstimate}
\mathbb{B}({\hat{\mb{E}}}_{w}) &=& \sum_{i \in {\cal I}} c_{i}\gamma_w(i)T^{-{i/2d}} + O\left(\frac{||w||_1}{\sqrt{T}}\right) \nonumber \\
&=& \sum_{i \in {\cal I}} c_{i}\gamma_w(i)T^{-{i/2d}} + O\left(\frac{\sqrt{L}||w||_2}{\sqrt{T}}\right) 
\end{eqnarray}


Denote the covariance matrix of $\{\hat{\mb{E}}_l; l \in \bar{l}\}$ by $\Sigma_L$. Let $\bar{\Sigma}_L = \Sigma_LT$. Observe that by (\ref{Variance_Ensemble}) and the Cauchy-Schwarz inequality, the entries of $\bar{\Sigma}_L$ are $O(1)$. The variance of the weighted estimator $\hat{\mb{E}}_w$ can then be bounded as follows:
\begin{eqnarray}
\label{eq:var}
\var(\hat{\mb{E}}_w) &=& \var(\sum_{l \in \bar{l}} w_l\hat{\mb{E}}_l) =  w' \Sigma_L w = \frac{w' \bar{\Sigma}_L w}{T}\nonumber \\
&\leq& \frac{\lambda_{\max}(\bar{\Sigma}_L)||w||^2_2}{T} \leq \frac{trace(\bar{\Sigma}_L)||w||^2_2}{T} \leq \frac{L||w||^2_2}{T} 
\end{eqnarray}

We seek a weight vector $w$ that (i) ensures that the bias of the weighted estimator is $O(T^{-1/2})$ and (ii) has low $\ell_2$ norm $||w||_2$ in order to limit the contribution of the variance, and the higher order bias terms of the weighted estimator.  To this end, let $w_{o}$ be the solution to the convex optimization problem defined in (2.3). The solution $w_o$ is the solution of 
\begin{equation}
\begin{aligned}
& \underset{w}{\text{minimize}}
& & ||w||^2_2 \\
& \text{subject to}
& & A_0w = b, 
\end{aligned}
\nonumber
\label{convexsol2}
\end{equation}
where $A_0$ and $b$ are defined below. Let $a_0$ be the vector of ones: $[1,1. .. ,1]_{1 \times L}$; and let $a_{i}$, for each $i \in \cal{I}$ be given by $a_{i} = [\psi_i(l_1), .. ,\psi_i(l_L)]$. Define $A_0 = [a'_0, a'_{i_1}, ... , a'_{i_I}]'$, $A_1 = [a'_{i_1}, ... , a'_{i_I}]'$ and $b = [1;0;0;..;0]_{(I+1) \times 1}$. 

 Since $L > I$, the system of equations $A_0 w = b$ is guaranteed to have at least one solution (assuming linear independence of the rows $a_i$).  The minimum squared norm $\eta_L(d) := ||w_0||_2^2$ is then given by
$$\eta_L(d) = {\frac{\text{det}(A_1A'_1)}{\text{det}(A_0A'_0)}}.$$ 

Consequently, by (\ref{Bias_EnsembleEstimate}), the bias $\mathbb{B}[\hat{\mb{E}}_{w_o}] = O(\sqrt{L\eta_L(d)}/\sqrt{T})$. By (\ref{eq:var}), the estimator variance $\var[\hat{\mb{E}}_{w_0}] = O(L\eta_L(d)/T)$. The overall MSE is also therefore of order $O(L\eta_L(d)/T)$.

For any fixed dimension $d$ and fixed number of estimators $L>I$ in the ensemble independent of sample size $T$, the value of $\eta_L(d)$ is also independent of $T$. Stated mathematically, $L\eta_L(d) = \Theta(1)$ for any fixed dimension $d$ and fixed number of estimators $L>I$ independent of sample size $T$. This concludes the proof.

\end{proof}

In the next section, we will verify conditions $\mathscr C.1$(\ref{BE}) and $\mathscr C.2$(\ref{VE}) for plug-in estimators $\hat{\mb{G}}_{k}(f)$ of entropy-like functionals $G(f) = \int g(f(x),x) f(x) dx$. 


\section{Application to estimation of functionals of a density}
\label{sec:entropyest}
Our focus is the estimation of general non-linear functionals $G(f)$ of $d$-dimensional multivariate densities $f$ with known finite support ${\cal S} = [a,b]^d$, where $G(f)$ has the form
\begin{equation}
\label{eq:oracle}
G(f) = \int g(f(x),x) f(x) dx, 
\end{equation}
for some smooth function $g(f,x)$. Let ${\cal B}$ denote the boundary of ${\cal S}$. Assume that $T=N+M$ i.i.d realizations  $\{\mb{X}_1, \ldots, \mb{X}_N, \mb{X}_{N+1}, \ldots, \mb{X}_{N+M}\}$ are available from the density $f$. 

\subsection{Plug-in estimators of entropy}
\label{sec:weightedpluginest}
The truncated \emph{uniform kernel} density estimator is defined below. For any positive real number $k \leq M$, define the distance $d_k$ to be: $d_k = (k/M)^{1/d}$. Define the truncated kernel region for each $X \in {\cal S}$ to be $S_k(X) = \{Y \in {\cal S} : ||X-Y||_\infty \leq d_k/2\}$, and the volume of the truncated uniform kernel to be $V_k(X) = \int_{S_k(X)} dz$. Note that when the smallest distance from $X$ to ${\cal B}$ is greater than $d_k/2$, $V_k(X) = d_k^d = k/M$. Let $\mb{l}_k(X)$ denote the number of samples falling in $S_k(X)$: $\mb{l}_k(X) = \sum_{i=1}^{M} 1_{\{\mb{X}_i \in S_k(X)\}}$. The truncated {uniform kernel} density estimator is defined as
\begin{equation}
  \hat{\mb{f}}_{k}(X) = \frac{\mb{l}_k(X)}{MV_k(X)}.
\end{equation}

The plug-in estimator of the density functional is constructed using a data splitting approach as follows. The data is randomly subdivided into two parts $\{\mb{X}_1, \ldots, \mb{X}_N\}$ and $\{\mb{X}_{N+1}, \ldots, \mb{X}_{N+M}\}$ of $N$ and $M$ points respectively. In the first stage, we form the kernel density estimate ${\hat{\mb{f}}_k}$ at the $N$ points $\{\mb{X}_1, \ldots, \mb{X}_N\}$ using the $M$ realizations $\{\mb{X}_{N+1}, \ldots, \mb{X}_{N+M}\}$. Subsequently, we use the $N$ samples $\{\mb{X}_1, \ldots, \mb{X}_N\}$ to approximate the functional $G(f)$ and obtain the plug-in estimator:
\begin{eqnarray}
\label{eq:plugin}
  \hat{\mb{G}}_k &=& \frac{1}{N}\sum_{i=1}^N g({\hat{\mb{f}}{_k}(\mb{X}_i)},\mb{X}_i). 
\end{eqnarray}
Also define a standard kernel density estimator $\tilde{\mb{f}}_k$, which is identical to $\hat{\mb{f}}_k$ except that the volume $V_k(X)$ is always set to the untruncated value $V_k(X) = k/M$. Define
\begin{eqnarray}
\label{eq:plugin2}
  \tilde{\mb{G}}_k &=& \frac{1}{N}\sum_{i=1}^N g({\tilde{\mb{f}}{_k}(\mb{X}_i)},\mb{X}_i). 
\end{eqnarray}
The estimator $\tilde{\mb{G}}_k$ is identical to the estimator of Gy{\"{o}}rfi and van der Meulen~\cite{gyrofi}. Observe that the implementation of $\tilde{\mb{G}}_k$, unlike $\hat{\mb{G}}_k$, does not require knowledge about the support of the density.

\subsubsection{Assumptions}
\label{sec:assump}  
We make a number of technical assumptions that will allow us to obtain tight MSE convergence rates for the kernel density estimators defined above. $({\cal {A}}.0)$ : Assume that $k = k_0M^\beta$ for some rate constant $0<\beta<1$, and assume that $M$, $N$ and $T$ are linearly related through the proportionality constant $\alpha_{frac}$ with: $0 < \alpha_{frac} < 1$, $M = \alpha_{frac}T$ and $N = (1-\alpha_{frac})T$. $({\cal {A}}.1)$ : Let the density $f$ be uniformly bounded away from $0$ and upper bounded on the set ${\cal S}$, i.e., there exist constants $\epsilon_0$, $\epsilon_\infty$ such that $0 < \epsilon_0 \leq f(x) \leq \epsilon_\infty < \infty$ $\forall x \in {\cal S}$. $({\cal {A}}.2)$: Assume that the density $f$ has continuous partial derivatives of order $d$ in the interior of the set ${\cal S}$, and that these derivatives are upper bounded. $({\cal {A}}.3)$: Assume that the function $g(f,x)$ has $\max\{\lambda,d\}$ partial derivatives w.r.t. the argument $f$, where $\lambda$ satisfies the condition  $\lambda \beta>1$. Denote the $n$-th partial derivative of $g(f,x)$ wrt $x$ by $g^{(n)}(f,x)$.  $({\cal {A}}.4)$: Assume that the absolute value of the functional $g(f,x)$ and its partial derivatives are strictly upper bounded in the range $\epsilon_0 \leq f \leq \epsilon_\infty$ for all $x$. $({\cal {A}}.5)$: Let $\epsilon \in (0,1)$ and $\delta \in (2/3,1)$. Let ${\cal C}(M)$ be a positive function satisfying the condition $ {\cal C}(M) = \Theta(\exp(-M^{\beta(1-\delta)}))$. For some fixed $0 < \epsilon< 1$, define $p_l = (1-\epsilon)\epsilon_0$ and $p_u = (1+\epsilon)\epsilon_\infty $. Assume that the conditions
$$ (i) \sup_{x}|h(0,x)|  < G_1 < \infty, $$
$$ (ii) \sup_{f \in (p_l,p_u),x}|h(f,x)|  < G_2 < \infty, $$
$$(iii) \sup_{f \in (1/k,p_u),x }|h(f,x)|{\cal C}(M)  < G_3 < \infty \hspace{0.15in} \forall M,$$ 
$$(iv)\sup_{f \in (p_l,2^dM/k),x} |h(f,x)|{\cal C}(M)  < G_4 < \infty \hspace{0.15in} \forall M,$$ are satisfied by 
$h(f,x) = g(f,x), g^{(3)}(f,x)$ and $g^{(\lambda)}(f,x)$, for some constants $G_1$, $G_2$, $G_3$ and $G_4$.

These assumptions are comparable to other rigorous treatments of entropy estimation. The assumption $({\cal {A}}.0)$ is equivalent to choosing the bandwidth of the kernel to be a fractional power of the sample size~\cite{raykar}. The rest of the above assumptions can be divided into two categories: (i) assumptions on the density $f$, and (ii) assumptions on the functional $g$. The assumptions on the smoothness, boundedness away from $0$ and $\infty$ of the density $f$ are similar to the assumptions made by other estimators of entropy as listed in Section II,~\cite{beir}. The assumptions on the functional $g$ ensure that $g$ is sufficiently smooth and that the estimator is bounded. These assumptions on the functional are readily satisfied by the common functionals that are of interest in literature: Shannon $g(f,x) = - \log(f)I(f>0) + I(f=0)$ and R\'enyi $g(f,x) = f^{\alpha-1}I(f>0) + I(f=0)$ entropy, where $I(.)$ is the indicator function, and the quadratic functional $g(f,x) = f^{2}$.


\subsubsection{Analysis of MSE}
\label{sec:mseanal}
Under the assumptions stated above, we have shown the following in the Appendix:
\begin{theorem}
\label{knnbiasH}
The biases of the plug-in estimators $\hat{\mb{G}}_{k}, \tilde{\mb{G}}_{k}$ are given by
\begin{eqnarray}
\label{Biasu}
\mathbb{B}(\hat{\mb{G}}_{k}) &=& \sum_{i=1}^{d} c_{1,i}\left({\frac{k}{M}}\right)^{i/d} + \frac{c_{2}}{k} \nonumber + o\left(\frac{1}{k} + \frac{k}{M}\right) \nonumber \\
\mathbb{B}(\tilde{\mb{G}}_{k}) &=& c_{1}\left({\frac{k}{M}}\right)^{1/d} + \frac{c_{2}}{k} \nonumber + o\left(\frac{1}{k} + \frac{k}{M}\right), \nonumber
\end{eqnarray}
where $c_{1,i}$, $c_1$ and $c_2$ are constants that depend on $g$ and $f$.
\end{theorem}

\begin{theorem}
\label{knnvarH}
The variances of the plug-in estimators $\hat{\mb{G}}_{k}, \tilde{\mb{G}}_{k}$ are identical up to leading terms, and are given by
\begin{eqnarray}
\label{Variance}
\var(\hat{\mb{G}}_{k}) &=& c_4\left(\frac{1}{N}\right)+ c_5\left(\frac{1}{M}\right) + o\left(\frac{1}{M} + \frac{1}{N}\right) \nonumber \\
\var(\tilde{\mb{G}}_{k}) &=& c_4\left(\frac{1}{N}\right)+ c_5\left(\frac{1}{M}\right) + o\left(\frac{1}{M} + \frac{1}{N}\right), \nonumber
\end{eqnarray}
where $c_4$ and $c_5$ are constants that depend on $g$ and $f$.
\end{theorem}

\subsubsection{Optimal MSE rate}

From Theorem \ref{knnbiasH}, observe that the conditions $k \to \infty$ and $k/M \to 0$ are necessary for the estimators $\hat{\mb{G}}_{k}$ and $\tilde{\mb{G}}_{k}$ to be unbiased. Likewise from Theorem \ref{knnvarH}, the conditions $N \to \infty$ and $M \to \infty$ are necessary for the variance of the estimator to converge to $0$. Below, we optimize the choice of bandwidth $k$ for minimum MSE, and also show that the optimal MSE rate is invariant to the choice of $\alpha_{frac}$.

\paragraph{Optimal choice of $k$}
Minimizing the MSE over $k$ is equivalent to minimizing the square of the bias over $k$. The optimal choice of $k$ is given by
\begin{eqnarray}
\label{kopt}
k_{opt} &=& \Theta({M^{{1}/{1+d}}}),
\end{eqnarray}
and the bias evaluated at $k_{opt}$ is  $\Theta({M^{{-1}/{1+d}}})$.

\paragraph{Choice of $\alpha_{frac}$}
Observe that the MSE of $\hat{\mb{G}}_{k}$ and $\tilde{\mb{G}}_{k}$ are dominated by the squared bias $(\Theta(M^{-2/(1+d)}))$ as contrasted to the variance $(\Theta(1/N+1/M))$. This implies that the asymptotic MSE rate of convergence is invariant to the selected proportionality constant  $\alpha_{frac}$.

In view of (a) and (b) above, the optimal MSE for the estimators $\hat{\mb{G}}{_k}$ and $\tilde{\mb{G}}{_k}$ is therefore achieved for the choice of $k=\Theta(M^{1/(1+d)})$, and is given by $\Theta(T^{-2/(1+d)})$. Our goal is to reduce the estimator MSE to $O(T^{-1})$. We do so by applying the method of weighted ensembles described in Section~\ref{sec:genmeth}.

\subsection{Weighted ensemble entropy estimator}
For a positive integer $L > I = d-1$, choose $\bar{l} = \{l_1, \ldots, l_L\}$ to be positive real numbers. Define the mapping $k(l) = l\sqrt{M}$ and let $\bar{k} =  \{k(l); l \in \bar{l}\}$. Define the weighted ensemble estimator 
\begin{equation}
\hat{\mb{G}}_w  = \sum_{l \in \bar{l}} w(l)\hat{\mb{G}}_{k(l)}.
\label{eq:weightedensemble}
\end{equation}
From Theorems \ref{knnbiasH} and \ref{knnvarH}, we see that the biases of the ensemble of estimators $\{\hat{\mb{G}}_{k(l)}; l \in \bar{l}\}$ satisfy ${\mathscr C}.1$(\ref{BE}) when we set  $\psi_i(l) = l^{i/d}$ and ${\cal I} = \{1,..,d-1\}$. Furthermore, the general form of the variance of $\hat{\mb{G}}_{k(l)}$ follows ${\mathscr C}.2$(\ref{VE}) because $N, M = \Theta(T)$. This implies that we can use the weighted ensemble estimator $\hat{\mb{G}}_w$ to estimate entropy at $O(L\eta_L(d)/T)$ convergence rate by setting $w$ equal to the optimal weight $w_o$ given by (\ref{convexsoltheory}).

\section{Experiments}
\label{sec:exp}

We illustrate the superior performance of the proposed weighted ensemble estimator for two applications: (i) estimation of the Panter-Dite rate distortion factor, and (ii) estimation of entropy to test for randomness of a random sample. 

For finite $T$ direct use of Theorem 1 can lead to excessively high variance. This is because forcing the condition (2.3) that $\gamma_w(i)=0$ is too strong and, in fact, not necessary. The careful reader may notice that to obtain $O(T^{-1}) $ MSE convergence  rate in Theorem 1 it is sufficient that $\gamma_w(i)$ be of order $O(T^{-1/2+i/2d})$. Therefore, in practice we determine the optimal weights according to the optimization:

\begin{equation}
\begin{split}
\min_w \quad &\epsilon\\
\text{subject to} \quad &\gamma_w(0) = 1,\\
&\lvert \gamma_w(i) T^{1/2 - i/2d} \rvert \leq \epsilon, \quad i \in \mathcal{I},\\
&\lVert w \rVert_2^2 \leq \eta.
\end{split}
\label{convexsol}
\end{equation}

The optimization \eqref{convexsol} is also  convex. Note that, as contrasted to \eqref{convexsoltheory}, the norm of the weight vector $w$ is bounded instead of being minimized.  By relaxing the constraints $\gamma_w(i)=0$ in \eqref{convexsoltheory} to the softer constraints in \eqref{convexsol}, the upper bound $\eta$ on $\lVert w \rVert_2^2$ can be reduced from the value $\eta_L(d)$ obtained by solving \eqref{convexsoltheory}.  This results in a more favorable trade-off between bias and variance for moderate sample sizes.  In our experiments, we find that setting $\eta = 3d$ yields good MSE performance.  Note that as $T \to \infty$, we must have $\gamma_w(i) \to 0$ for $i \in \mathcal{I}$  in order to keep $\epsilon$ finite, thus recovering the strict constraints in \eqref{convexsoltheory}.
%



For fixed sample size $T$ and dimension $d$, observe that increasing $L$ increases the number of degrees of freedom in the convex problem \eqref{convexsol}, and therefore will result in a smaller value of $\epsilon$ and in turn improved estimator performance.   In our simulations, we choose $\bar{l}$ to be $L=50$ equally spaced values between $0.3$ and $3$, ie the $l_i$ are uniformly spaced as $$l_i = \frac{x}{a} + \frac{(a-1)ix}{aL}; i=1,..,L,$$ with scale and range parameters $a=10$ and $x=3$ respectively. We limit $L$ to 50 because we find that the gains beyond $L=50$ are negligible. The reason for this diminishing return is a direct result of the increasing similarity among the entries in $\bar{l}$, which translates to increasingly similar basis functions $\psi_i(l) = l^{i/d}$. 



\subsection{Panter-Dite factor estimation}

\begin{figure}[!t]
\centering
\subfigure[\small{Variation of MSE of Panter-Dite factor estimates as a function of sample size $T$. From the figure, we see that the proposed weighted estimator has the fastest MSE rate of convergence wrt sample size $T$ ($d=6$).}]{
\includegraphics[scale=.25]{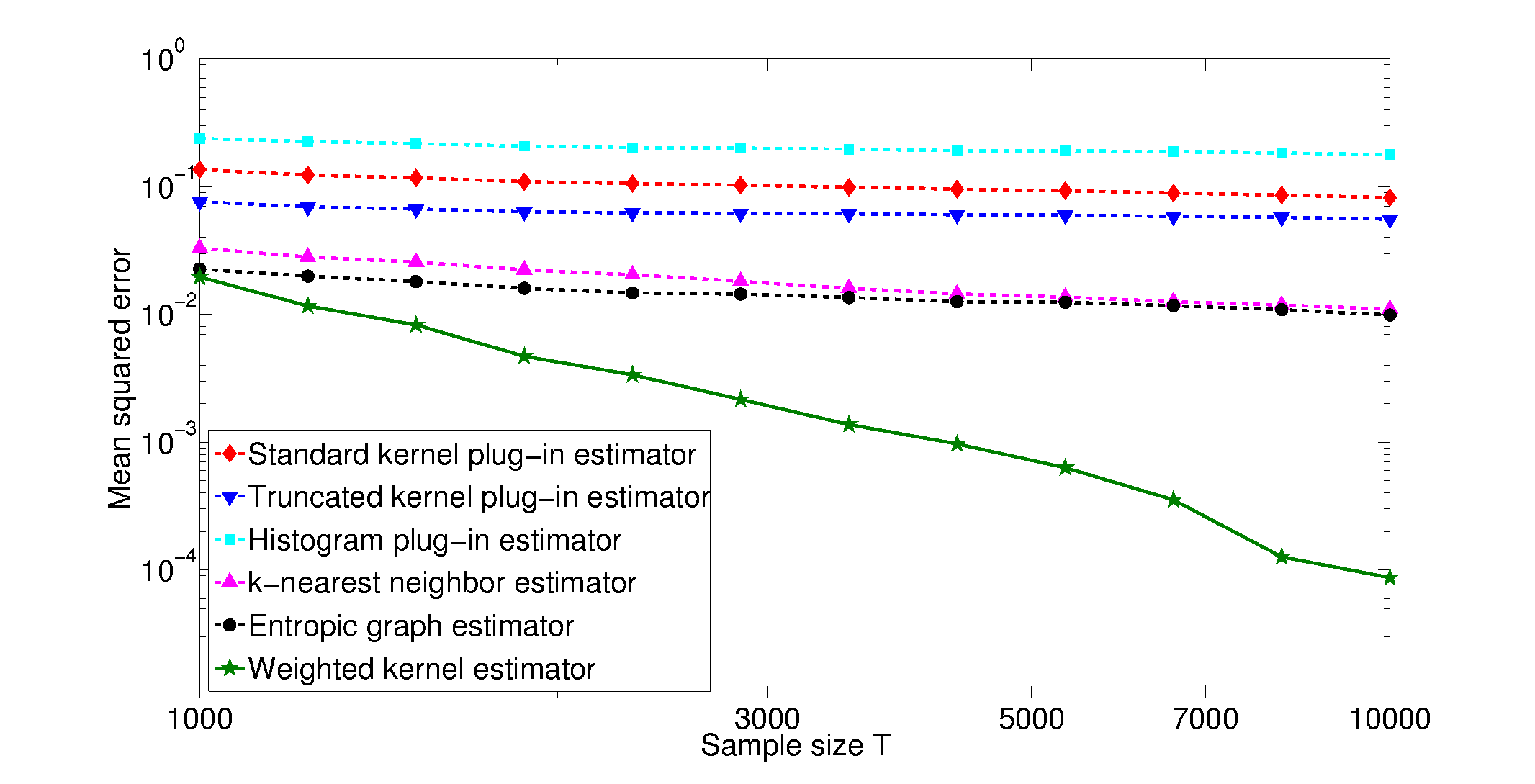}
\label{a-compare}
}
\subfigure[\small{Variation of MSE of Panter-Dite factor estimates as a function of dimension $d$. From the figure, we see that the MSE of the proposed weighted estimator has the slowest rate of growth with increasing dimension $d$ ($T=3000$).}]{
\includegraphics[scale=.25]{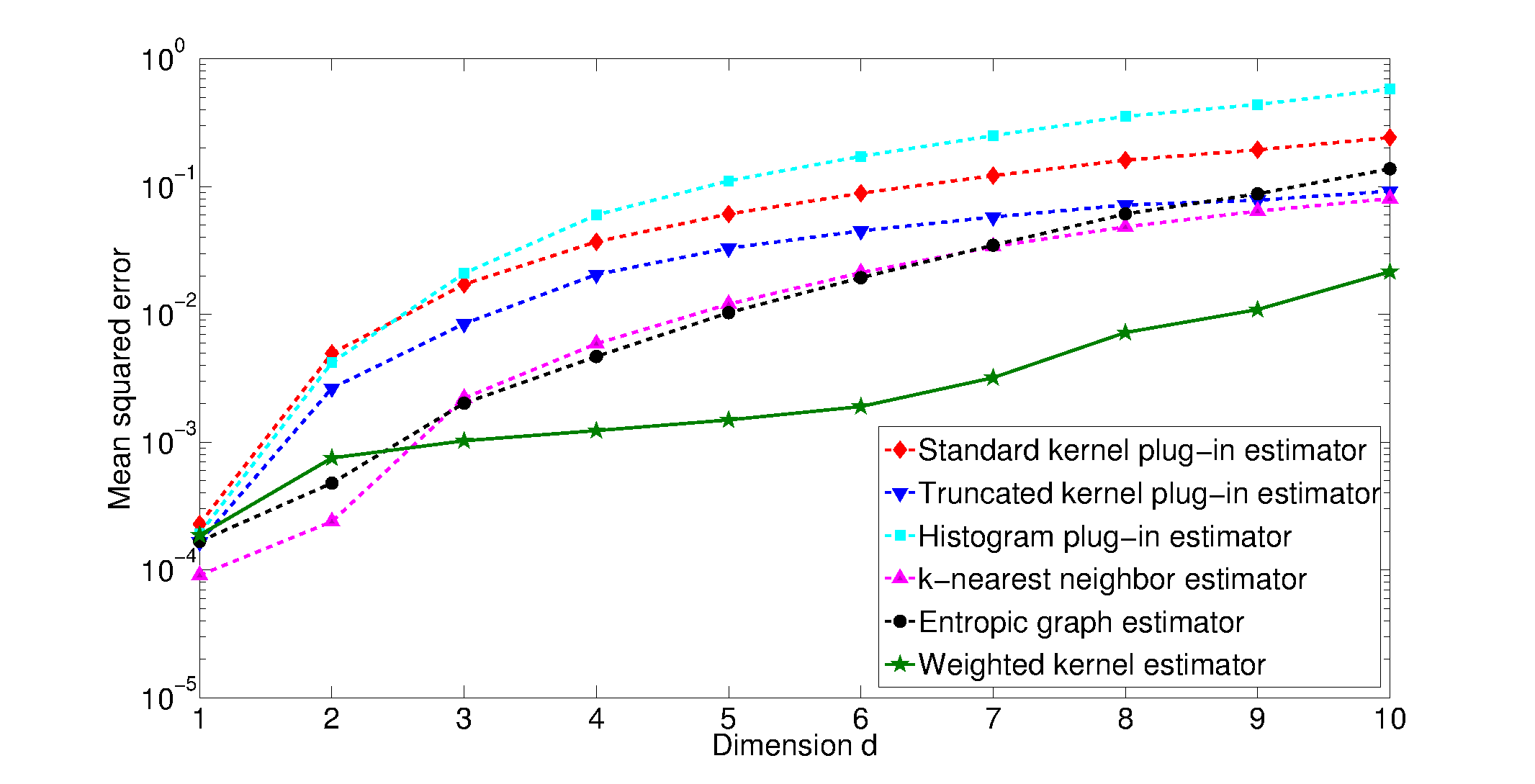}
\label{ad-compare}
}
\caption{Variation of MSE of Panter-Dite factor estimates using standard kernel plug-in estimator~[14], truncated kernel plug-in estimator~(3.3), histogram plug-in estimator[17], $k$-NN estimator~[20], entropic graph estimator~[18] and the weighted ensemble estimator~(3.6). }
\end{figure}

For a $d$-dimensional source with underlying density $f$, the Panter-Dite distortion-rate function~\cite{riten} for a $q$-dimensional vector quantizer with $n$ levels of quantization is given by $ \delta(n)= n^{-2/q} \int f^{q/(q+2)}(x) dx. $ The Panter-Dite factor corresponds to the functional $G(f)$ with $g(f,x) = n^{-2/q} f^{-2/(q+2)}I(f>0) + I(f=0)$. The Panter-Dite factor is directly related to the R\'{e}nyi $\alpha$-entropy, for which several other estimators have been proposed~\cite{gyrofi0,heroJ,pal,leo2}. 

In our simulations we compare six different choices of functional estimators - the three estimators previously introduced: (i) the standard kernel plug-in estimator $\tilde{\mb{G}}_k$, (ii) the boundary truncated plug-in estimator $\hat{\mb{G}}_k$ and (iii) the weighted estimator $\hat{\mb{G}}_w$ with optimal weight $w = w^*$ given by (\ref{convexsol}), and in addition the following popular entropy estimators: (iv) histogram plug-in estimator~\cite{gyrofi0}, (v) $k$-nearest neighbor ($k$-NN) entropy estimator~\cite{leo2} and (vi) entropic $k$-NN graph estimator~\cite{heroJ,pal}. For both $\tilde{\mb{G}}_k$ and $\hat{\mb{G}}_k$, we select the bandwidth parameter $k$ as a function of $M$ according to the optimal proportionality $k=M^{1/(1+d)}$ and $N=M=T/2$. 

We choose $f$ to be the $d$ dimensional mixture density $f(a,b,p,d) = pf_\beta(a,b,d) + (1-p)f_u(d)$; where $d=6$, $f_\beta(a,b,d)$ is a $d$-dimensional Beta density with parameters $a=6,b=6$, $f_u(d)$ is a $d$-dimensional uniform density and the mixing ratio $p$ is $0.8$. The reason we choose the beta-uniform mixture for our experiments is because it trivially satisfies all the assumptions on the density $f$ listed in Section 3.1, including the assumptions of finite support and strict boundedness away from 0 on the support. The true value of the Panter-Dite factor $\delta(n)$ for the beta-uniform mixture is calculated using numerical integration methods via the 'Mathematica' software (http://www.wolfram.com/mathematica/). Numerical integration is used because evaluating the entropy in closed form for the beta-uniform mixture is not tractable. 

The MSE values for each of the six estimators are calculated by averaging the squared error $[\hat{\delta}_i(n) - \delta(n)]^2$, $i=1,..,m$ over $m = 1000$ Monte-Carlo trials, where each $\hat{\delta}_i(n)$ corresponds to an independent instance of the estimator.


\subsubsection{Variation of MSE with sample size $T$}
The MSE results of the different estimators are shown in Fig.~\ref{a-compare} as a function of sample size $T$, for fixed dimension $d=6$. It is clear from the figure that the proposed ensemble estimator $\hat{\mb{G}}_w$ has significantly faster rate of convergence while the MSE of the rest of the estimators, including the truncated kernel plug-in estimator, have similar, slow rates of convergence. It is therefore clear that the proposed optimal ensemble averaging significantly accelerates the MSE convergence rate.


\subsubsection{Variation of MSE with dimension $d$}

For fixed sample size $T$ and fixed number of estimators $L$, it can be seen that $\epsilon$ increases monotonically with $d$. This follows from the fact that the number of constraints in the convex problem \ref{convexsol} is equal to $d+1$ and each of the basis functions $\psi_i(l) = l^{i/d}$ monotonically approaches $1$ as $d$ grows, . This in turn implies that for a fixed sample size $T$ and number of estimators $L$, the overall MSE of the ensemble estimator should increase monotonically with the dimension $d$.

The MSE results of the different estimators are shown in Fig.~\ref{ad-compare} as a function of dimension $d$, for fixed sample size $T=3000$. For the standard kernel plug-in estimator and truncated kernel plug-in estimator, the MSE increases  rapidly with $d$ as expected. The MSE of the histogram and $k$-NN estimators increase at a similar rate, indicating that these estimators suffer from the curse of dimensionality as well. On the other hand, the MSE of the weighted estimator also increases with the dimension as predicted, but at a slower rate. Also observe that the MSE of the weighted estimator is smaller than the MSE of the other estimators for all dimensions $d>3$.

\subsection{Distribution testing}

\begin{figure}[!t]
\centering
\subfigure[\small{Entropy estimates for random samples corresponding to hypothesis $H_0$ (experiments 1-500) and $H_1$ (experiments 501-1000).}]{
\includegraphics[scale=.30]{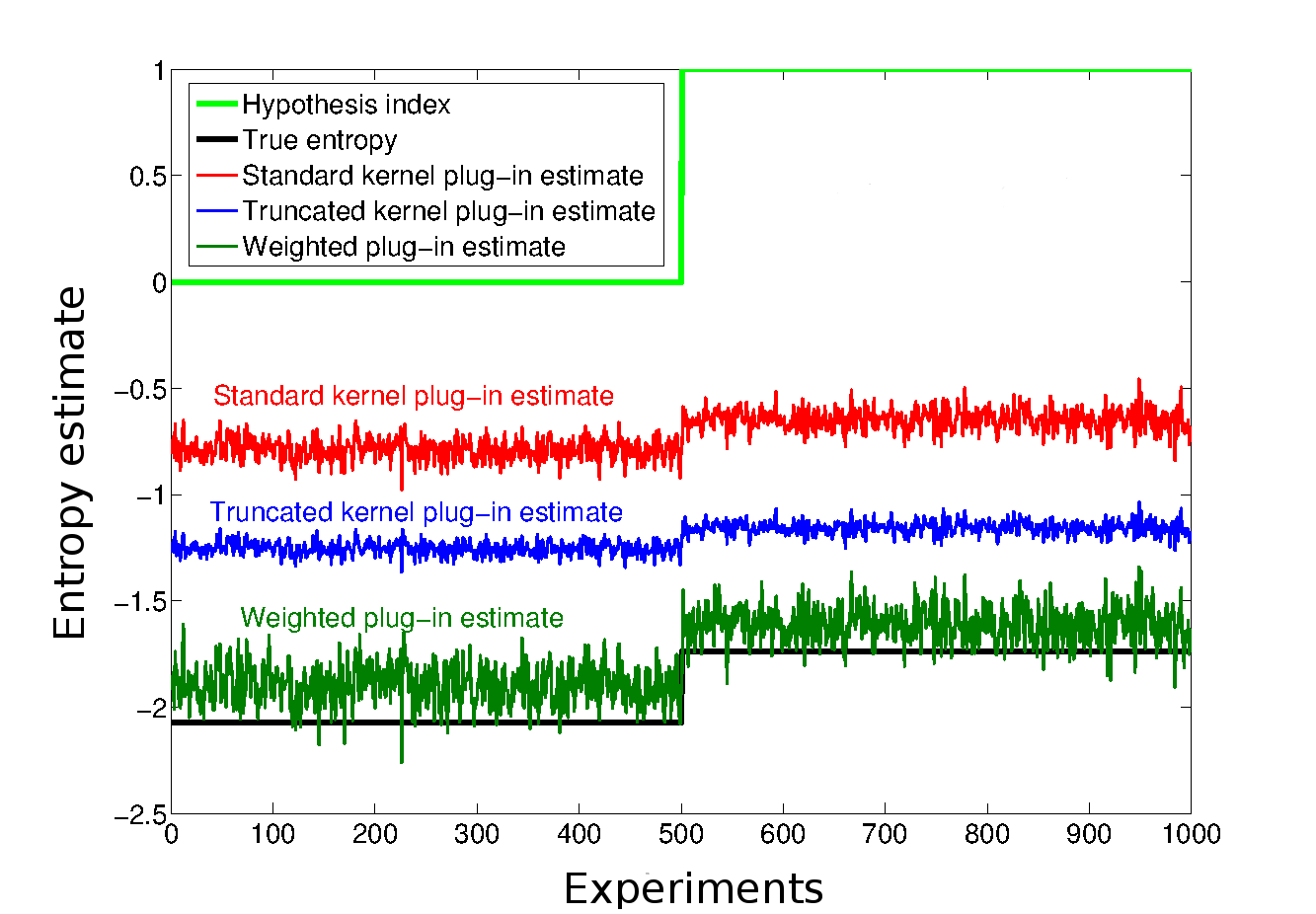}
\label{c-compare}
}
\subfigure[\small{Histogram envelopes of entropy estimates for random samples corresponding to hypothesis $H_0$ (blue) and $H_1$ (red).}]{
\includegraphics[scale=.30]{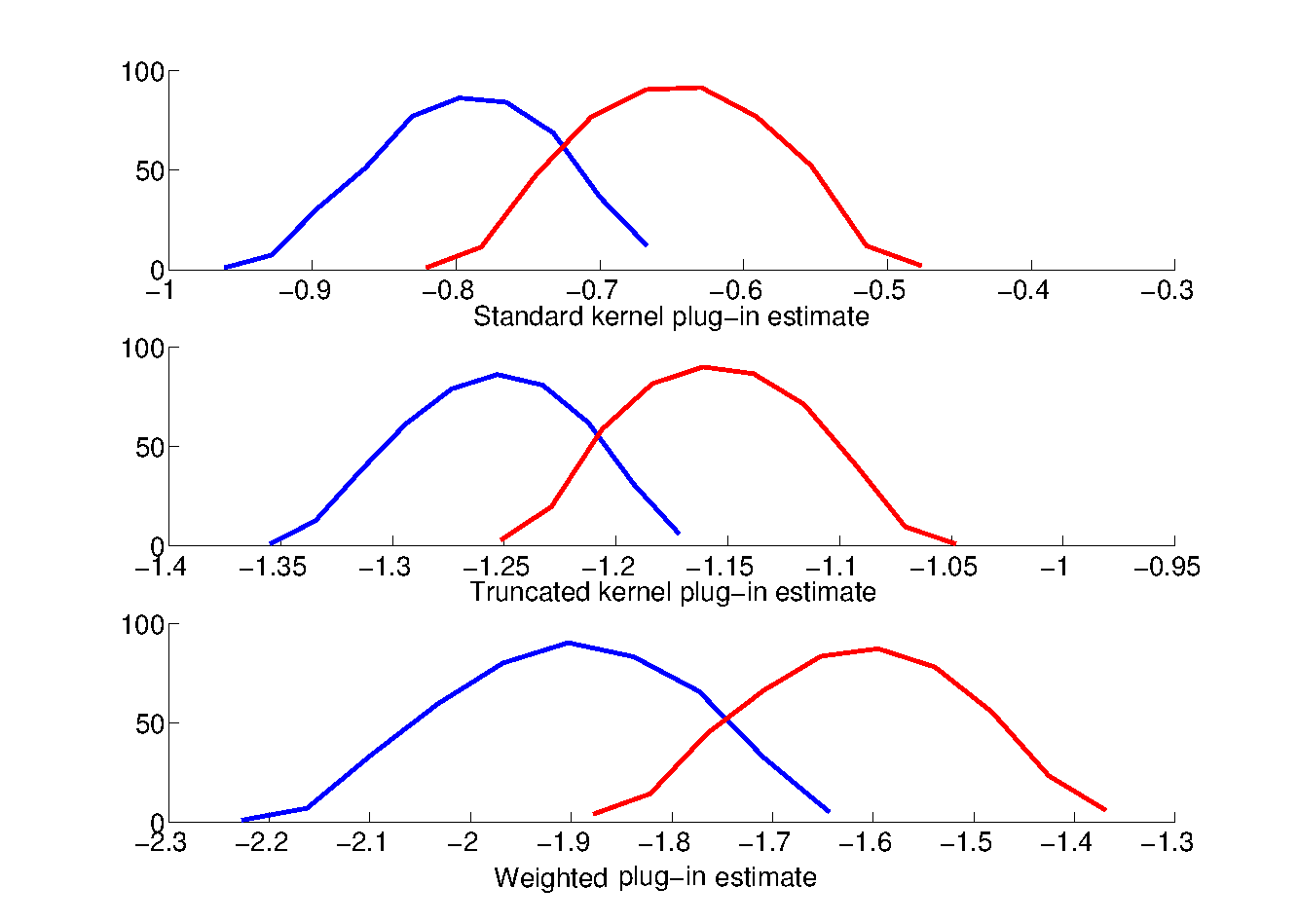}
\label{d-compare}
}
\caption{Entropy estimates using standard kernel plug-in estimator, truncated kernel plug-in estimator and the weighted estimator, for random samples corresponding to hypothesis $H_0$ and $H_1$. The weighted estimator provides better discrimination ability by suppressing the bias, at the cost of some additional variance.}
\end{figure}

In this section, we illustrate the weighted ensemble estimator for non-parametric estimation of Shannon differential entropy.  The Shannon differential entropy is given by $G(f)$ where $g(f,x) = - \log(f)I(f>0) + I(f=0)$. The improved accuracy of the weighted ensemble estimator is demonstrated in the context of hypothesis testing using estimated entropy as a statistic to test for the underlying probability distribution of a random sample. Specifically, the samples under the null and alternate hypotheses $H_0$ and $H_1$ are drawn from the probability distribution $f(a,b,p,d)$, described in Section IV.A, with fixed $d=6$, $p=0.75$ and two sets of values of $a,b$ under the null and alternate hypothesis, $H_0: a=a_0,b=b_0$ versus $H_1: a=a_1,b=b_1$.

First, we fix $a_0=b_0=6$ and $a_1=b_1=5$. The density under the null hypothesis $f(6,6,0.75,6)$ has greater curvature relative to $f(5,5,0.75,6)$ and therefore has smaller entropy.   Five hundred (500) experiments are performed under each hypothesis with each experiment consisting of 1000 samples drawn from the corresponding distribution. The true entropy and  estimates $\tilde{\mb{G}}_k$, $\hat{\mb{G}}_k$ and $\hat{\mb{G}}_w$ obtained from each instance of $10^3$ samples are shown in Fig.~\ref{c-compare} for the 1000 experiments. This figure suggests that the  ensemble weighted estimator provides better discrimination ability by suppressing the bias, at the cost of some additional variance.

To demonstrate that the weighted estimator provides better discrimination, we plot the histogram envelope of the entropy estimates using standard kernel plug-in estimator, truncated kernel plug-in estimator and the weighted estimator for the cases corresponding to the hypothesis $H_0$ (color coded {blue}) and $H_1$ (color coded {red}) in Fig.~\ref{d-compare}. Furthermore, we quantitatively measure the discriminative ability of the different estimators using the deflection statistic $ds = {|\mu_1-\mu_0|}/{\sqrt{\sigma_0^2+\sigma_1^2}},$ where $\mu_0$ and $\sigma_0$ (respectively  $\mu_1$ and $\sigma_1$) are the sample mean and standard deviation of the entropy estimates. The deflection statistic was found to be $1.49$, $1.60$ and $1.89$ for the standard kernel plug-in estimator, truncated kernel plug-in estimator and the weighted estimator respectively. The receiver operating curves (ROC) for this entropy-based test using the three different estimators are shown in Fig.~\ref{b-compare}. The corresponding areas under the ROC curves (AUC) are given by $0.9271$, $0.9459$ and $0.9619$. 

\begin{figure}[!t]
\centering
\subfigure[\small{ROC curves corresponding to entropy estimates obtained using standard and truncated kernel plug-in estimators and the weighted estimator. The corresponding AUC are given by $0.9271$, $0.9459$ and $0.9619$.}]{
\includegraphics[scale=.30]{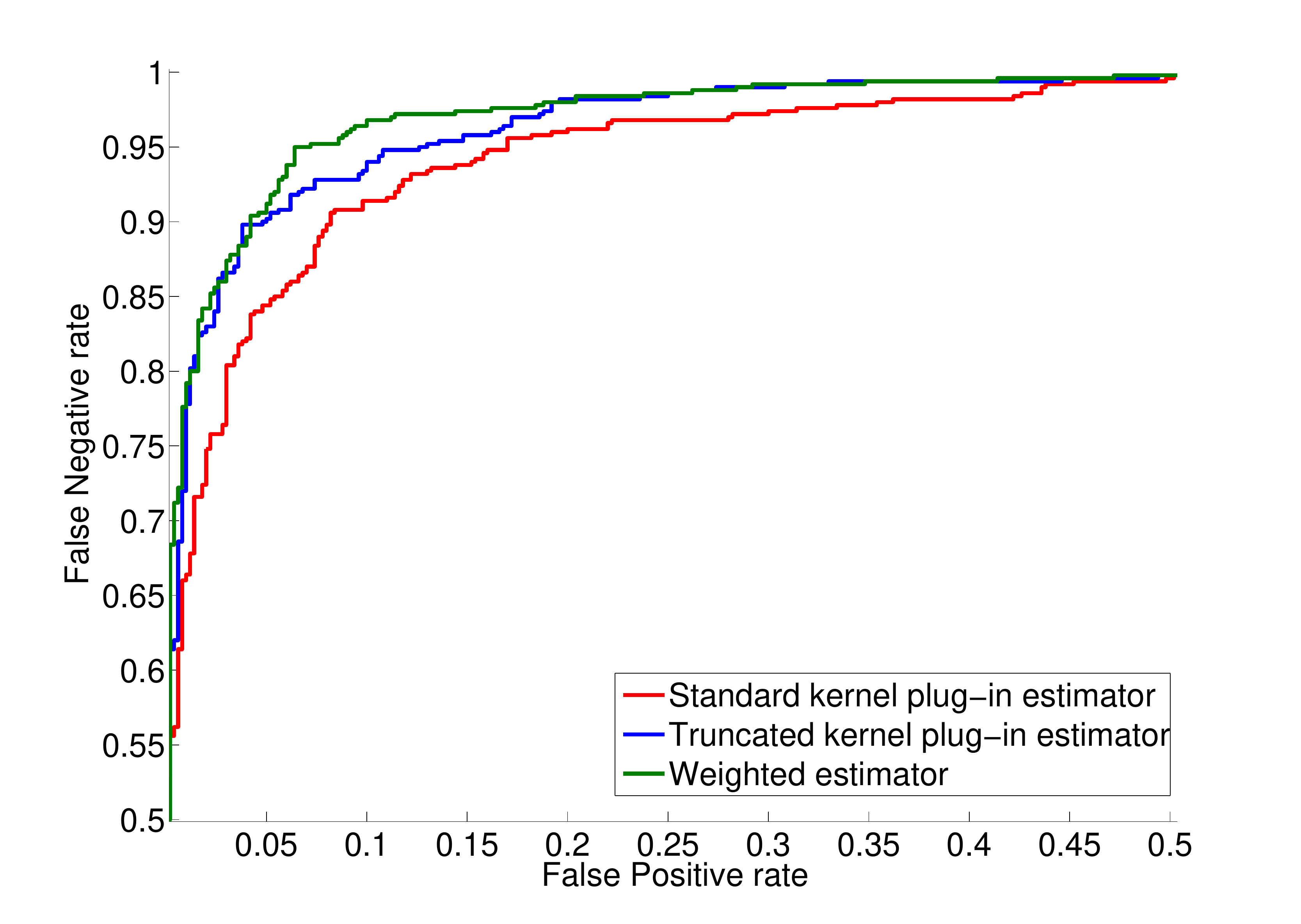}
\label{b-compare}
}
\subfigure[\small{Variation of AUC curves vs $\delta (= a_0-a_1, b_0-b_1)$ corresponding to Neyman-Pearson omniscient test, entropy estimates using the standard and truncated kernel plug-in estimators and the weighted estimator.}]{
\includegraphics[scale=.30]{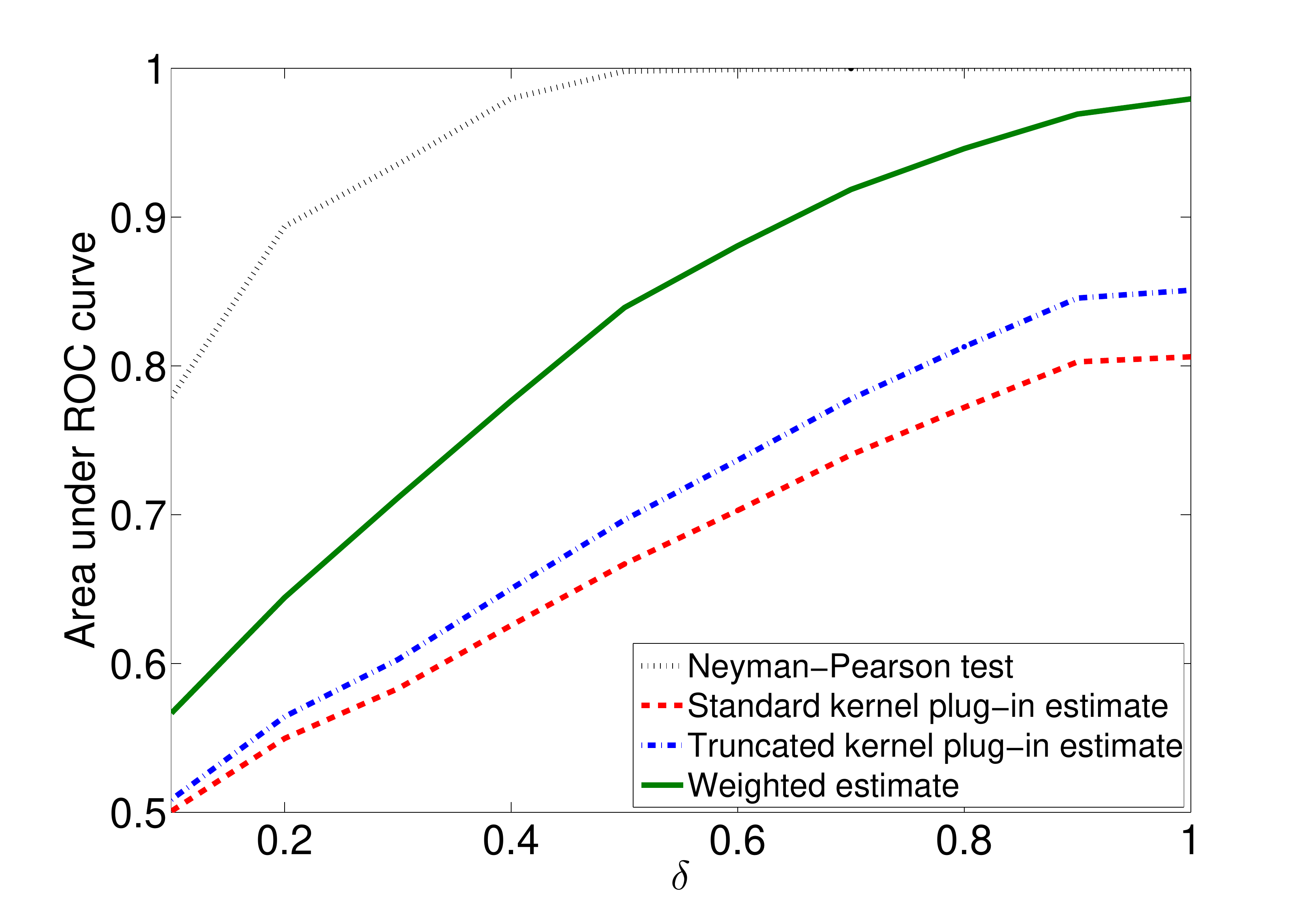}
\label{e-compare}
}
\caption{Comparison of performance in terms of ROC for the distribution testing problem. The weighted estimator uniformly outperforms the individual plug-in estimators.}
\end{figure}

In our final experiment, we fix $a_0=b_0=10$ and set $a_1=b_1=10-\delta$,   perform 500 experiments each under the null and alternate hypotheses with samples of size 5000, and plot the AUC as $\delta$ varies from $0$ to $1$ in Fig.~\ref{e-compare}. For comparison, we also plot the AUC for the Neyman-Pearson likelihood ratio test. The Neyman-Pearson likelihood ratio test, unlike the Shannon entropy based tests, is an omniscient test that assumes knowledge of both the underlying beta-uniform mixture parametric model of the density and the parameter values $a_0$, $b_0$ and $a_1$, $b_1$ under the null and alternate hypothesis respectively. Figure 4 shows that the weighted estimator {\emph {uniformly and significantly}} outperforms the individual plug-in estimators and comes closest to the performance of the omniscient Neyman-Pearson likelihood test. The relatively superior performance of the Neyman-Pearson likelihood test is due to the fact that the weighted estimator is a nonparametric estimator that has marginally higher variance (proportional to $||w^*||_2^2$) as compared to the underlying parametric model for which the Neyman-Pearson test statistic provides the most powerful test. 

\section{Conclusions}
\label{sec:con}
We have proposed a new estimator of functionals of a multivariate density based on weighted ensembles of kernel density estimators.  For ensembles of estimators that satisfy general conditions on bias and variance as specified by ${\mathscr C}.1$(\ref{BE}) and ${\mathscr C}.2$(\ref{VE}) respectively,  the weight optimized ensemble estimator has parametric $O(T^{-1})$ MSE convergence rate that can be much faster than the rate of convergence of any of the individual estimators in the ensemble. The optimal weights are determined as a solution to a convex optimization problem that can be performed offline and does not require training data. We illustrated this estimator for  uniform kernel plug-in estimators and demonstrated the superior performance of the weighted ensemble entropy estimator for (i)  estimation of the Panter-Dite factor and (ii) non-parametric hypothesis testing. 

Several extensions of the framework of this paper are being pursued: (i) using $k$-nearest neighbor ($k$-NN) estimators in place of kernel estimators; (ii) extending the framework to the case where support ${\cal S}$ {\emph {is not known}}, but for which conditions ${\mathscr C}.1$ and ${\mathscr C}.2$ hold; (iii) using ensemble estimators for estimation of other functionals of probability densities including divergence, mutual information and intrinsic dimension; and (iv) using an $l_1$ norm $\|w\|_1$in place of the $l_2$ norm $\|w\|_2$ in  the weight optimization algorithm (2.3) so as to introduce sparsity into the weighted ensemble. 

\section*{Acknowledgement}
This work was partially supported by (i) ARO grant W911NF-12-1-0443 and (ii) NIH grant 2P01CA087634-06A2.

\begin{figure}[h]
\centering
\includegraphics[width=6.2in]{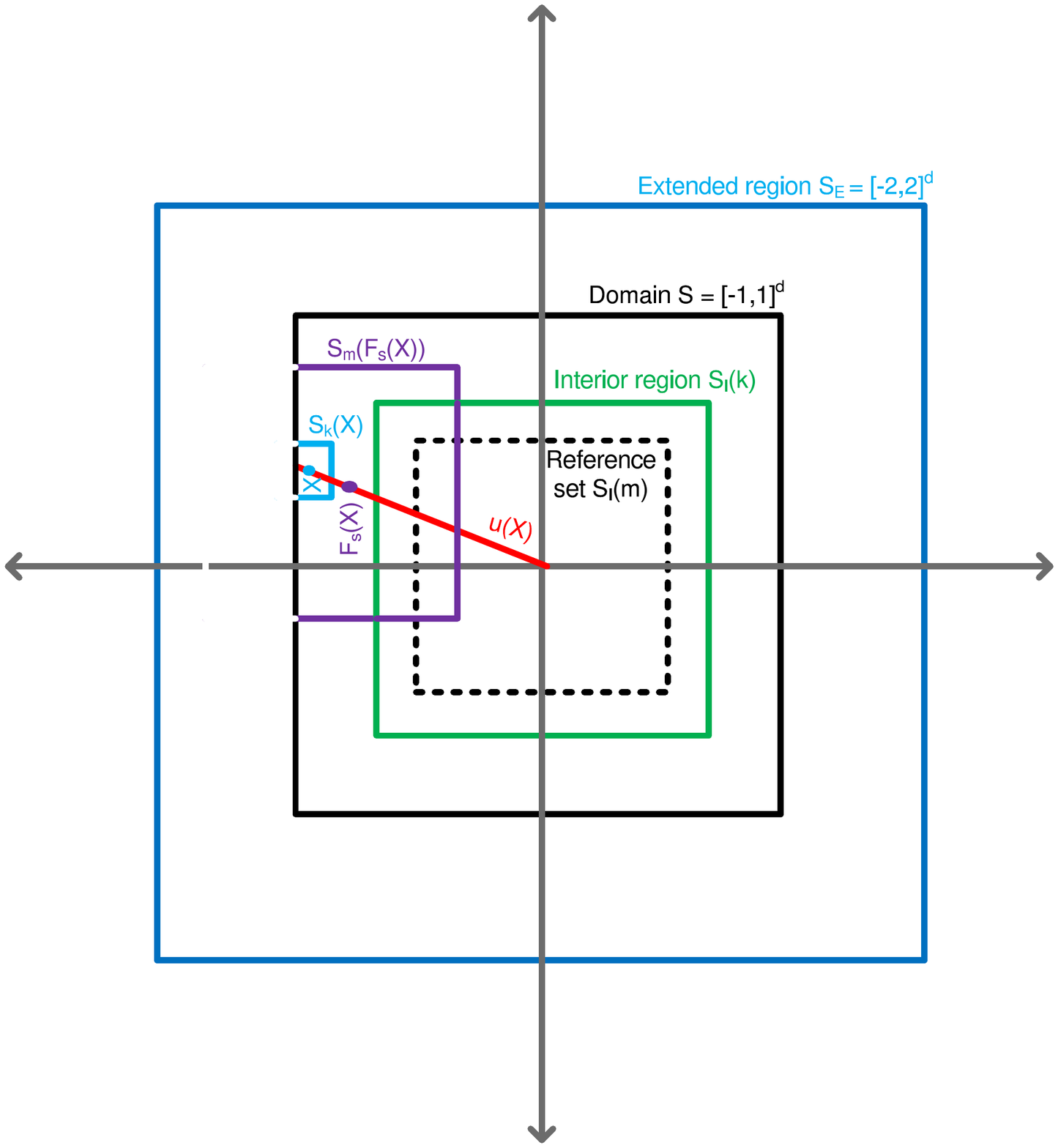}
\caption{Illustration for the proof of Lemma~\ref{biaslemma}.}
\label{i-compare}
\end{figure}

\appendix
\appendixpage

\section*{Outline of appendix}
We first establish moment properties for uniform kernel density estimates in Appendix~\ref{sec:kernmoments}. Subsequently, we prove theorems \ref{knnbiasH} and \ref{knnvarH} in Appendix~\ref{sec:biasvarproof}.

\section{Moment properties of boundary compensated uniform kernel density estimates}
\label{sec:kernmoments}
Throughout this section, we assume without loss of generality that the support ${\cal S} = [-1,1]^d$. Observe that $\mb{l}_k(X)$ is a binomial random variable with parameters $M$ and $U_k(X) = Pr(\mb{Z} \in S_k(X))$. The probability mass function of the binomial random variable $\mb{l}_k(X)$ is given by
\begin{equation}
Pr(\mb{l}_k(X)=l) = \binom{M}{l} (U_k(X))^{l}(1-U_k(X))^{M-l}. \nonumber \\
\end{equation} 

Define the error function of the truncated uniform kernel density,
\begin{eqnarray}
\hat{\mb{e}}_k(X) &=& \hat{\mb{f}}_{k}(X)-\expect[\hat{\mb{f}}_{k}(X)] \nonumber \\
&=& \frac{l_k(X)}{MV_k(X)} - \frac{U_k(X)}{V_k(X)} \nonumber \\
&=& \frac{\sum_{i=1}^M (1_{\mb{X}_i \in S_k(X)} - U_k(X))}{MV_k(X)}.
\end{eqnarray}

Also define the error function of the standard uniform kernel density,
\begin{eqnarray}
\tilde{\mb{e}}_k(X) &=& \tilde{\mb{f}}_{k}(X)-\expect[\tilde{\mb{f}}_{k}(X)] \nonumber \\
&=& (MV_k(X)/k)\hat{\mb{e}}_k(X), \nonumber
\end{eqnarray}
and note that when $X \in {\cal S}_I(k)$, $\tilde{\mb{e}}_k(X) = \hat{\mb{e}}_k(X)$.

\subsection{Taylor series expansion of {coverage}}


For any $X \in {\cal S}$, the {coverage} function $U_k(X)$ can be represented by using a $d$ order Taylor series expansion of $f$ about $X$ as follows. Because the density $f$ has continuous partial derivatives of order $d$ in ${\cal S}$, for any $X \in {\cal S}$,
\begin{eqnarray}
U_{k}(X) &=& \int_{S_{k}(X)}f(z) dz \nonumber \\
&=& f(X)V_{k}(X) + \sum_{i=1}^{d} c_{i,k}(X)V_{k}^{1+i/d}(X) + o((k/M)^2),
\end{eqnarray}
where $c_{i,k}$ are functions which depend on $k$ and the unknown density $f$. This implies that the expectation of the density estimate is given by
\begin{eqnarray}
\expect{[\hat{\mb{f}}_{k}(X)]} &=& U_k(X)/V_k(X)  \nonumber \\
&=& f(X) + \sum_{i=1}^{d} c_{i,k}(X){\left(\frac{k}{M}\right)}^{i/d} + o\left({\left(\frac{k}{M}\right)}\right).
\label{inbias}
\end{eqnarray}

\subsection{Concentration inequalities for uniform kernel density estimator}
Because $\mb{l}_k(X)$ is a binomial random variable, standard Chernoff inequalities can be applied to obtain concentration bounds on $\mb{l}_k(X)$. In particular, for $0<p<1/2$,
\begin{eqnarray}
Pr({\mb{l}_k(X)>(1+p)MU_k(X)}) \leq e^{-MU_k(X)p^2/4},  \nonumber \\
Pr({\mb{l}_k(X)<(1-p)MU_k(X)}) \leq e^{-MU_k(X)p^2/4}.
\end{eqnarray}
Let $\natural(X)$ denote the event $(1-p_k)MU_k(X)<\mb{l}_k(X)<(1+p_k)MU_k(X)$, where $p_k = 1/(k^{\delta/2})$ for some fixed $\delta \in (2/3,1)$. Then, for $k = O(M^\beta)$, 
\begin{equation}
Pr(\natural^c(X)) = O(e^{-p_k^2 k}) = {\cal C}(M),
\end{equation}
where ${\cal C}(M)$ satisfies the condition $ \lim_{M \to \infty} M^a/{\cal C}(M) = 0$ for any $a>0$. Also observe that under the event $\natural(X)$, 
\begin{eqnarray}
\hat{\mb{e}}_k(X) &=& \frac{l_k(X)}{MV_k(X)} - \frac{U_k(X)}{V_k(X)} \nonumber \\
&=& O(p_kU_k(X)/V_k(X)) = O(p_k) = O(1/(k^{\delta/2})).
\end{eqnarray}

\subsection{Bounds on uniform kernel density estimator}
Let $B_r(X)$ be an Euclidean ball of radius $r$ centered at $X$. Let $X$ be a Lebesgue point of $f$, i.e., an $X$ for which $$\lim_{r \to 0} \frac{\int_{B_r(X)} f(y) dy}{\int_{B_r(x)} dy}  = f(X).$$ Because $f$ is an density, we know that almost all $X \in {\cal S}$ satisfy the above property. Now, fix $\epsilon \in (0,1)$ and find $\epsilon_r > 0$ such that 
$$\sup_{0<r\leq \epsilon_r} \frac{\int_{B_r(X)} f(y) dy}{\int_{B_r(x)} dy} - f(X) \leq \epsilon/2 f(X).$$ For small values of $k/M$, $B_{\epsilon_r}(X) \subset S_k(X)$ and therefore
\begin{eqnarray}
{(1-\epsilon/2)f(X)V_k(X)} \leq U_k(X) \leq {(1+\epsilon/2)f(X)V_k(X)}
\label{coverageineqnatural}
\label{coverageineq}
\end{eqnarray}
This implies that under the event $\natural(X)$  defined in the previous subsection,
\begin{eqnarray}
(1-\epsilon) \epsilon_0 \leq &\hat{\mb{f}}_{k}(X)& \leq (1+\epsilon)\epsilon_\infty.
\label{densityineqnatural}
\end{eqnarray}
Let $\natural_0(X)$ denote the event that $\hat{\mb{f}}_{k}(X) = 0$. Let $\natural_1(X)$ denote the event $1<=\mb{l}_k(X)<=(1-p_k)MU_k(X)$ and $\natural_2(X)$ denote $\mb{l}_k(X)>=(1+p_k)MU_k(X)$. Then conditioned on the event  
$\natural_1(X)$ 
\begin{eqnarray}
1/k \leq &\hat{\mb{f}}_{k}(X)& \leq (1+\epsilon)\epsilon_\infty.
\label{densityineqnaturalc1}
\end{eqnarray}
and conditioned on the event $\natural_2(X)$ 
\begin{eqnarray}
(1-\epsilon) \epsilon_0 \leq &\hat{\mb{f}}_{k}(X)& \leq 2^dM/k.
\label{densityineqnaturalc2}
\end{eqnarray}
Observe that $\natural_0(X)$, $\natural_1(X)$, $\natural_2(X)$ and $\natural(X)$ form a disjoint partition of the event space.

\subsection{Bias}

\begin{lemma}
\label{biaslemma}
Let $\gamma(x,y)$  be an arbitrary function with $d$ partial derivatives wrt $x$ and $\sup_{x,y} |\gamma(x,y)| < \infty$. Let $\mb{X}_{1},..,\mb{X}_{M},\mb{X}$ denote $M+1$ i.i.d realizations of the density $f$. Then,
\begin{equation}
\expect[\gamma(\check{{f}}_{k}(\mb{Z}),\mb{Z})] - \expect[\gamma({{f}}(\mb{Z}),\mb{Z})] = \sum_{i=1}^{d} c_{1,i}(\gamma(x,y))(k/M)^{i/d} + o((k/M)),
\label{bknnbias}
\end{equation}
where $c_{1,i}(\gamma(x,y))$ are functionals of $\gamma$ and $f$.
\end{lemma}

\begin{proof}
To analyze the bias, first extend the density function $f$ as follows. In particular, extend the definition of $f$ to the domain ${\cal S}_E = [-2,2]^d$ while ensuring that the extended function $f_e$ is differentiable $d$ times on this extended domain. Let $s_k(X) = \{Y: ||X-Y||_1 \leq d_k/2\}$ be the natural un-truncated ball. Let $u_k(X) = \int_{z \in s_k(X)} f_e(z) dz$. Define the function $\bar{f}_k(X) = u_k(X)/(k/M)$. For any $X \in {\cal S}$, using this extended definition,
\begin{eqnarray}
u_{k}(X) &=& \int_{s_{k}(X)}f_e(z) dz \nonumber \\
&=& f(X)(k/M) + \sum_{i=1}^{d} c_{i}(X)(k/M)^{1+i/d} + o((k/M)^2),
\end{eqnarray}
where $c_{i}$ are only functions of the unknown density $f_e$. Also define $\check{f}_k(X) = \expect[\mb{\hat{f}}_k(X) \mid X]$. Define the interior region ${\cal S}_I(k) = \{X \in {\cal S}: s_k(X) \cap {\cal S}^c = \phi\}$. Note that $\bar{f}_k(X) = \check{f}_k(X)$ for all $X \in {\cal S}_I(k)$. Now,
\begin{eqnarray}
\expect[\gamma(\check{{f}}_{k}(\mb{Z}),\mb{Z})] - \expect[\gamma({{f}}(\mb{Z}),\mb{Z})]  &=& \expect[\gamma(\bar{f}_k(\mb{Z}),\mb{Z}) - \gamma({f}(\mb{Z}),\mb{Z})] + \expect[\gamma(\check{{f}}_k(\mb{Z}),\mb{Z}) - \gamma(\bar{f}_k(\mb{Z}),\mb{Z})] \nonumber \\
&=& \expect[\gamma(\bar{f}_k(\mb{Z}),\mb{Z}) - \gamma({f}(\mb{Z}),\mb{Z})] + \expect[1_{\mb{Z} \in {\cal S-S}_I(k)}
 ( \gamma(\check{{f}}_k(\mb{Z}),\mb{Z}) - \gamma(\bar{f}_k(\mb{Z}),\mb{Z}))] \nonumber \\
&=& I+II.
\end{eqnarray}

\subsubsection{Evaluation of I}:
\begin{eqnarray}
I &=& \expect[\gamma(\bar{f}_k(\mb{Z}),\mb{Z}) - \gamma({f}(\mb{Z}),\mb{Z})] \nonumber \\
&=& \sum_{i=1}^d \expect \left[\gamma^{(i)}({f}(\mb{Z}),\mb{Z}) {\left(\bar{f}_k(\mb{Z}) - {f}(\mb{Z})\right)^i} \right] \nonumber \\
&=& \sum_{i=1}^{d} c_{11,i}(\gamma(x,y))(k/M)^{i/d} + o((k/M)),
\end{eqnarray}
where $c_{11,i}(\gamma(x,y))$ are functionals of $\gamma(x,y)$ and its derivatives.
\newcommand{\fs}{{\cal F}_s}
\newcommand{\fb}{{\cal F}_b}
\newcommand{\fr}{{\cal F}_r}

\subsubsection{Evaluation of II}

Let $m=M/2$, $k_M = k/M$ and $k_m = (k/m)^{1/d}$. Define mappings ${\cal F}_b$, $\fr$ and $\fs$: ${\cal S - S}_I(k)$ $\to$ ${\cal B}$ as follows. Let $u(X)$ denote the unit vector from the origin to $X$, and define ${\cal F}_b(X) = u(X) \cap {\cal B}.$ Let ${\cal S}_I(m)$ be a reference set. Define ${\cal F}_r(X) = u(X) \cap {\cal S}_I(m).$ Let $l_b(X) = ||{\cal F}_b(X) - X||$. Finally define ${\cal F}_s(X) = n(X)u(X),$ where $n(X)$ satisfies $||{\cal F}_b(X) - {\cal F}_s(X)|| = (m/k)^{1/d} l_b(X)$. For each $X \in {\cal S - S}_I(k)$, let $l_r(X) = ||{\cal F}_b(X) - {\cal F}_s(X)||$ and $l_{max}(X) = ||{\cal F}_b(X) - {\cal F}_r(X)||$. Let ${\cal U}$ denote the set of all unit vectors: ${\cal U} = \cup_{\{X \in {\cal S - S}_I(k)\}} u(X).$  Observe that, by definition, the {\emph {shape}} of the regions $S_k(X)$ and $S_m(\fs(X))$ is identical. This is illustrated in Fig.~\ref{i-compare}.


\paragraph{Analysis of $\bar{{f}}_{m}(\fs(X))$, $\check{{f}}_{m}(\fs(X))$ }
$\fb(X)$ can represented in terms of $\fs(X)$ as $\fb(X) = \fs(X) + l_s(X)u(X)$. Using Taylor series around ${\cal F}_b(X)$, $\check{{f}}_{m}({\cal F}_s(X))$ can then be evaluated as
\begin{eqnarray}
{\check{f}}_{m}(\fs(X)) &=& U_{m}(\fs(X))/V_{m}(\fs(X)) \nonumber \\
&=& f(\fb(X)) + \sum_{i=1}^{d} \grave{c}_{i,\fb(X)}(\fb(X))l_r^{i}(X) + o(l_r^d(X)),
\end{eqnarray}
where the functionals $\grave{c}_{i,\fb(X)}$ depend only on the {{shape}} of the regions $S_k(X)$ or $S_m(\fs(X))$ and therefore only on $\fb(X)$. Similarly, 
\begin{eqnarray}
\bar{{f}}_{m}(X) &=& u_{m}(X)/(1/2) \nonumber \\
&=& f(\fb(X)) + \sum_{i=1}^{d} \acute{c}_{i,\fb(X)}(\fb(X))l_r^{i}(X) + o(l_r^d(X)),
\end{eqnarray}
where the functionals $\acute{c}_{i,\fb(X)}$ again depend only on $\fb(X)$. This implies that for any fixed $u \in {\cal U}$ and corresponding $X_b \in {\cal B}$, for any function $\eta(x)$ and positive integer $q \in \{1,..,d\}$, integration over the line $l(X_b) = \{ X_b-cu(X_b);c \in (0,l_{max}(X_b)) \}$
\begin{align}
&\int_{Z \in l(X_b)} \eta(Z) (\check{{f}}_{m}(Z)-{f}(Z))^q dZ \nonumber \\
& = \sum_{i=q}^{d} \grave{c}_{i,q,\eta}(X_b)l_{max}^{i}(X) + o(l_{max}^d(X)),
\end{align}
and 
\begin{align}
&\int_{Z \in l(X_b)} \eta(Z) (\bar{{f}}_{m}(Z)-{f}(Z))^r dZ \nonumber \\
& = \sum_{i=q}^{d} \acute{c}_{i,q,\eta}(X_b)l_{max}^{q}(X) + o(l_{max}^d(X)),
\end{align}
where the functions $\acute{c}_{i,q,\eta}(X_b)$ and $\grave{c}_{i,q,\eta}(X_b)$ depend only on $X_b$, $q$, $\eta$ and are independent of $Z$ and $k$.

\paragraph{Analysis of $\bar{{f}}_{k}(X)$, $\check{{f}}_{k}(X)$ }
$\fb(X)$ can be represented in terms of $X$ as $\fb(X) = X + k_ml_r(X)u(X)$. Identically, this gives, 
\begin{eqnarray}
\check{{f}}_{k}(X) &=& U_{k}(X)/V_{k}(X) \nonumber \\
&=& f(\fb(X)) + \sum_{i=1}^{d} \grave{c}_{i,\fb(X)}(\fb(X))k^i_ml_r^{i}(X) + o(k^d_Ml_r^d(X)).
\end{eqnarray}
and
\begin{eqnarray}
\bar{{f}}_{k}(X) &=& u_{k}(X)/k_M \nonumber \\
&=& f(\fb(X)) + \sum_{i=1}^{d} \acute{c}_{i,\fb(X)}(\fb(X))k^i_Ml_r^{i}(X) + o(k^d_Ml_r^d(X)).
\end{eqnarray}

This implies that for any fixed $u \in {\cal U}$ and corresponding $X_b \in {\cal B}$, integration over the line $l(X_b) = \{ X_b-cu(X_b);c \in (0,k_ml_{max}(X_b)) \}$
\begin{align}
&\int_{Z \in l(X_b)} \eta(Z) (\check{{f}}_{k}(Z)-{f}(Z))^q dZ \nonumber \\
& = \sum_{i=q}^{d} \grave{c}_{i,r,\eta}(X_b)k^i_Ml_{max}^{i}(X) + o(k^d_ml_{max}^d(X)),
\end{align}
and
\begin{align}
&\int_{Z \in l(X_b)} \eta(Z) (\bar{{f}}_{k}(Z)-{f}(Z))^q dZ \nonumber \\
& = \sum_{i=q}^{d} \acute{c}_{i,r,\eta}(X_b)k^i_Ml_{max}^{i}(X) + o(k^d_ml_{max}^d(X)).
\end{align}

\paragraph{Analysis of II}
\begin{align}
II & = \expect[1_{\mb{Z} \in {\cal S-S}_I(k)}
 ( \gamma(\check{{f}}_{k}(\mb{Z}),\mb{Z}) - \gamma(\bar{f}_k(\mb{Z}),\mb{Z}))] \nonumber \\
& = \int_{{Z} \in {\cal S-S}_I(k)} ( \gamma(\check{{f}}_{k}({Z}),{Z}) - \gamma(\bar{f}_k({Z}),{Z})) f(Z) dZ \nonumber \\
& = \int_{\{X_b \in {\cal B}\}\cup\{c \in (0,k_ml_{max}(X_b))\}}1_{\{Z = X_b-cu(X_b)\}} \sum_{i=1}^d \left[\gamma^{(i)}({f}(X_b),X_b) {\left(\check{f}_k(Z) - {f}({Z})\right)^i} \right] f(Z) dZ \nonumber \\
& - \int_{\{X_b \in {\cal B}\}\cup\{c \in (0,k_ml_{max}(X_b))\}}1_{\{Z = X_b-cu(X_b)\}} \sum_{i=1}^d \left[\gamma^{(i)}({f}(X_b),X_b) {\left(\bar{f}_k(Z) - {f}({Z})\right)^i} \right] f(Z) dZ \nonumber \\
&=  \sum_{i=1}^{d} c_{12,i}(\gamma(x,y))(k/M)^{i/d} + o((k/M)),
\end{align}
where $c_{12,i}(\gamma(x,y))$ are functionals of $\gamma(x,y)$ and its derivatives. This implies that 
\begin{eqnarray}
\expect[\gamma(\check{{f}}_{k}(\mb{Z}))] - \expect[\gamma({{f}}(\mb{Z}))]  &=& I+II \nonumber \\
&=&  \sum_{i=1}^{d} c_{1,i}(\gamma(x,y))(k/M)^{i/d} + o((k/M)),
\label{bknnbias2}
\end{eqnarray}
where the functionals $c_{1,i}(\gamma(x,y))$ are independent of $k$.
\end{proof}

\subsection{Central Moments}
Since $\mb{l}_k(X)$ is a binomial random variable, we can easily obtain moments of the uniform kernel density estimate in terms of $U_k(X)$. These are listed below.

\begin{lemma}
\label{momentlemma}
Let $\gamma(x)$  be an arbitrary function satisfying $\sup_{x} |\gamma(x)| < \infty$. Let $\mb{X}_{1},..,\mb{X}_{M},\mb{X}$ denote $M+1$ i.i.d realizations of the density $f$. Then,
\begin{equation}
\expect{\left[\gamma(\mb{X})\hat{\mb{e}}^q_k(\mb{X})\right]} = {1_{\{q=2\}}}c_2(\gamma(x))\left(\frac{1}{k}\right) + o\left(\frac{1}{k}\right),
\label{bknncent}
\end{equation}
\begin{equation}
\expect{\left[\gamma(\mb{X})\tilde{\mb{e}}^q_k(\mb{X})\right]} = {1_{\{q=2\}}}c_2(\gamma(x))\left(\frac{1}{k}\right) + o\left(\frac{1}{k}\right),
\label{bknncent2}
\end{equation}
where $c_2(\gamma(x))$ is a functional of $\gamma$ and $f$.
\end{lemma}

\begin{proof}
When $r=2$,
\begin{eqnarray}
\var[\hat{\mb{f}}_{k}(X)] &=& \expect{[\hat{\mb{e}}^2_{k}(X)]}  \nonumber \\
&=& \frac{U_k(X)(1-U_k(X))}{MV^2_k(X)} \nonumber \\
&=& \frac{f(X)}{MV_k(X)} + o\left(\frac{1}{k}\right).
\label{secmom}
\end{eqnarray}

For any integer $r\geq3$, 
\begin{eqnarray}
\label{fourmom}
\expect{[\hat{\mb{e}}^r_{k}(X)]} &=& \expect{[1_{\natural(X)}\hat{\mb{e}}^r_{k}(X)]} + \expect{[1_{\natural^c(X)}\hat{\mb{e}}^r_{k}(X)]}\nonumber \\
&=& O\left(\frac{1}{k^{\delta r/2}}\right) = o(1/k).
\end{eqnarray}

Observe that $V_k(X) = \Theta(k/M)$ and therefore $\expect{[\hat{\mb{e}}^2_{k}(X)]} = \Theta(1/k) + o(1/k)$. This implies,
\begin{equation}
\expect{\left[\gamma(\mb{X})\hat{\mb{e}}^q_k(\mb{X})\right]} = {1_{\{q=2\}}}c_2(\gamma(x))\left(\frac{1}{k}\right) + o\left(\frac{1}{k}\right). \nonumber 
\end{equation}


When $X \in {\cal S}_I(k)$, $\tilde{\mb{e}}_k(X) = \hat{\mb{e}}_k(X)$. Also $Pr(\mb{X} \in {\cal S}_I(k)) = o(1)$. This result in conjunction with the fact that $\tilde{e}_k(X) = (MV_k(X)/k)\hat{e}_k(X)$, and $V_k(X) = \Theta(k/M)$ gives
\begin{equation}
\expect{\left[\gamma(\mb{X})\tilde{\mb{e}}^q_k(\mb{X})\right]} = {1_{\{q=2\}}}c_2(\gamma(x))\left(\frac{1}{k}\right) + o\left(\frac{1}{k}\right). \nonumber 
\end{equation}

\end{proof}

\subsection{Cross moments}

Let $X$ and $Y$ be two distinct points. Clearly the density estimates at $X$ and $Y$ are not independent. Observe that the uniform kernel regions $S_k(X)$, $S_k(Y)$ are disjoint for the set of points given by $\Psi_k := \{X,Y\}:||X-Y||_1 \geq 2(k/M)^{1/d}$, and have finite intersection on the complement of $\Psi_k$. 
\newcommand{\he}{\hat{\mb{e}}}
\newcommand{\te}{\tilde{\mb{e}}}

\paragraph{Intersecting balls}
\begin{lemma}
\label{intersectu}
For a fixed pair of points $\{X,Y\} \in \Psi_k$, and positive integers $q,r$,
\begin{equation}
\label{highercrossmoments}
Cov[\he^q_k(X),\he^r_k(Y)] = 1_{\{q=1,r=1\}}\left(\frac{-f(X)f(Y)}{M}\right) +o\left(\frac{1}{M}\right). \nonumber 
\end{equation}
\end{lemma}

\begin{proof}
For a fixed pair of points $\{X,Y\}\in {\Psi_K}$, the joint probability mass function of the functions $\mb{l}_k(X)$,$\mb{l}_k(Y)$ is given by
\begin{equation}
Pr(\mb{l}_k(X)=l_x,\mb{l}_k(Y)=l_y) = 1_{\{l_x+l_y \leq M\}}\binom{M}{l_x,l_y} (U_k(X))^{l_x}(U_k(Y))^{l_y}(1-U_k(X)-U_k(Y))^{M-l_x-l_y}. \nonumber
\end{equation}
%
%
Denote the high probability event $\natural(X) \cap \natural(Y)$ by $\natural(X,Y)$. Define $\mb{\hat{l}}_{k}(X)$, $\mb{\hat{l}}_k(Y)$ to be binomial random variables with parameters $\{U_k(X)$,$M-q\}$ and $\{U_k(Y)$,$M-r\}$ respectively. The covariance between powers of density estimates is then given by 
\begin{eqnarray}
&& Cov(\hat{\mb{f}}_{k}^q(X),\hat{\mb{f}}_{k}^r(Y)) = \left(\frac{1}{M^{q+r}V^q_k(X)V^r_k(Y)}\right)Cov(\mb{l}_{k}^q(X),\mb{l}_{k}^r(Y)) \nonumber \\
&& = \left(\frac{1}{M^{q+r}V^q_k(X)V^r_k(Y)}\right) \sum l_x^ql_y^r \left[ Pr(\mb{l}_k(X)=l_x,\mb{l}_k(Y)=l_y) - Pr(\mb{l}_k(X)=l_x)Pr(\mb{l}_k(Y)=l_y)\right] \nonumber \\
&& = \left(\frac{1}{M^{q+r}V^q_k(X)V^r_k(Y)}\right) \sum_{\natural(X,Y)} {l_x^ql_y^r} \left[Pr(\mb{l}_k(X)=l_x,\mb{l}_k(Y)=l_y) - Pr(\mb{l}_k(X)=l_x)Pr(\mb{l}_k(Y)=l_y)\right] + o\left(\frac{1}{M}\right) \nonumber \\
&& = \left(\frac{1}{M^{q+r}V^q_k(X)V^r_k(Y)}\right) \sum_{\natural(X,Y)} \frac{l_x^ql_y^rU_k^q(X)U_k^r(Y)}{(l_x \times \ldots \times l_x-{q+1})(l_y \times \ldots \times l_y-{r+1})} \times \nonumber \\
& & \Bigl[(M \times \ldots \times M-(q+r-1))Pr(\mb{\hat{l}}_k(X)=l_x,\mb{\hat{l}}_k(Y)=l_y) \nonumber \\
&& - (M \times \ldots \times M-q+1)(M \times \ldots \times M-r+1) Pr(\mb{\hat{l}}_k(X)=l_x)Pr(\mb{\hat{l}}_k(Y)=l_y) \Bigr]  + o\left(\frac{1}{M}\right) \nonumber \\
&& =\left(\frac{f^q(X)f^r(Y)}{M^{q+r}}\right)  \times \nonumber \\
&& \sum_{\natural(X,Y)} \Bigl[(M \times \ldots \times M-(q+r-1)) Pr(\mb{\hat{l}}_k(X)=l_x,\mb{\hat{l}}_k(Y)=l_y) \nonumber \\
&& - (M \times \ldots \times M-(q-1))(M \times \ldots \times M-(r-1)) Pr(\mb{\hat{l}}_k(X)=l_x)Pr(\mb{\hat{l}}_k(Y)=l_y) \Bigr] + o\left(\frac{1}{M}\right) \nonumber \\
&& =\left(\frac{f^q(X)f^r(Y)}{M^{q+r}}\right)  \times \nonumber \\
&& [(M \times \ldots \times M-(q+r-1)) - (M \times \ldots \times M-(q-1))(M \times \ldots \times M-(r-1)) ] \nonumber \\
&& =\frac{-qrf^q(X)f^r(Y)}{M}+o\left(\frac{1}{M}\right). \nonumber
\end{eqnarray}

Then, the covariance between the powers of the error function is given by 
\begin{eqnarray}
Cov(\he^q_k(X),\he^r_k(Y)) &=& Cov((\hat{\mb{f}}_k(X)-\expect[\hat{\mb{f}}_k(X)])^q,(\hat{\mb{f}}_k(Y)-\expect[\hat{\mb{f}}_k(Y)])^r) \nonumber \\
&=& \sum_{a=1}^{q} \sum_{b=1}^{r} \binom{q}{a} \binom{r}{b} (-\expect[\hat{\mb{f}}_k(X)])^a(-\expect[\hat{\mb{f}}_k(Y)])^b Cov(\hat{\mb{f}}_k^a(X),\hat{\mb{f}}_k^b(Y)) \nonumber \\
&=& \sum_{a=1}^{q} \sum_{b=1}^{r} \binom{q}{a} \binom{r}{b} [(-f(X))^a(-f(Y))^b+o(1)] Cov(\hat{\mb{f}}_k^a(X),\hat{\mb{f}}_k^b(Y)) \nonumber \\
&=& -f^{q}(X)f^{r}(Y) \sum_{a=1}^{q} \sum_{b=1}^{r} \binom{q}{a} \binom{r}{b}  \frac{(-1)^{a+b}ab}{M}+o\left(\frac{1}{M}\right) \nonumber \\
&=& 1_{\{q=1,r=1\}}\left(\frac{-f(X)f(Y)}{M}\right) +o\left(\frac{1}{M}\right). \nonumber 
\end{eqnarray}
\end{proof}

\paragraph{Disjoint balls}

For $\{X,Y\} \in \Psi_k^c$, there is no closed form expression for the covariance. However we have the following lemma by applying the Cauchy-Schwartz inequality:
\begin{lemma}
\label{disjointu}
For a fixed pair of points $\{X,Y\}\in \Psi_k^c$, 

\begin{equation}
Cov[\he_k^q(X),\he_k^r(Y)] = {1_{\{q=1,r=1\}}}O\left(\frac{1}{k}\right) + o\left(\frac{1}{k}\right). \nonumber
\end{equation}

\end{lemma}

\begin{proof}
\begin{align}
|Cov[\he_k^q(X),\he_k^r(Y)]| &\leq \sqrt{\var[\he_k^{2q}(X)]\var[\he_k^{2r}(Y)]} \nonumber \\
& {1_{\{q=1,r=1\}}}O\left(\frac{1}{k}\right) + o\left(\frac{1}{k}\right). \nonumber
\end{align}
\end{proof}

\paragraph{Joint expression}
\begin{lemma}
\label{covariancelemma}
Let $\gamma_1(x)$, $\gamma_2(x)$ be arbitrary functions with $1$ partial derivative wrt $x$ and $\sup_{x} |\gamma_1(x)| < \infty$, $\sup_{x} |\gamma_2(x)| < \infty$. Let $\mb{X}_{1},..,\mb{X}_{M},\mb{X},\mb{Y}$ denote $M+2$ i.i.d realizations of the density $f$. Then,
\begin{equation}
Cov{\left[\gamma_1(\mb{X})\he^q_k(\mb{X}),\gamma_2(\mb{Y})\he^q_k(\mb{Y})\right]} = {1_{\{q=1,r=1\}}}c_{5}(\gamma_1(x),\gamma_2(x))\left(\frac{1}{M}\right) + o\left(\frac{1}{M}\right),
\label{bknncross}
\end{equation}
\begin{equation}
Cov{\left[\gamma_1(\mb{X})\te^q_k(\mb{X}),\gamma_2(\mb{Y})\te^q_k(\mb{Y})\right]} = {1_{\{q=1,r=1\}}}c_5(\gamma_1(x),\gamma_2(x))\left(\frac{1}{M}\right) + o\left(\frac{1}{M}\right),
\label{bknncross2}
\end{equation}
where $c_5(\gamma_1(x),\gamma_2(x))$ is a functional of $\gamma_1(x)$, $\gamma_2(x)$ and $f$.
\end{lemma}

\begin{proof}
Let the indicator function ${1_{\Delta_k}}(X,Y)$ denote the event ${\Delta_k}: \{{X},{Y}\}\in \Psi_k^c$. Then
\begin{eqnarray}
Cov{\left[\gamma_1(\mb{X})\he^q_k(\mb{X}), \gamma_2(\mb{Y})\he^r_k(\mb{Y})\right]} = I + D, \nonumber
\end{eqnarray}
where '$I$' stands for the contribution form the intersecting balls and '$D$' for the contribution from the dis-joint balls. $I$ and $D$ are given by
\begin{eqnarray}
I &=& \expect{\left[\mb{1_{\Delta_k}}(\mb{X},\mb{Y}) Cov \left[\gamma_1({X})\he^q_k({X}), \gamma_2({Y}) \he^r_k({Y}) \right]\right]},\nonumber \\
D &=& \expect{\left[\mb{(1-\mb{1_{\Delta_k}}(\mb{X},\mb{Y}))} Cov \left[\gamma_1({X})\he^q_k({X}), \gamma_2({Y}) \he^r_k({Y}) \right]\right]}. \nonumber 
\end{eqnarray}

When $1_{\Delta_k}({X},{Y}) \neq 0$, we have $\{X,Y\} \in \Psi_k^c$. Then,
\begin{eqnarray}
I &=& \expect{\left[\mb{1_{\Delta_k}}(\mb{X},\mb{Y}) \gamma_1(\mb{X})\gamma_2(\mb{Y}) \he^q_k(\mb{X})\he^r_k(\mb{Y})\right]}  \nonumber \\
&=& \expect{\left[\mb{1_{\Delta_k}}(\mb{X},\mb{Y}) \gamma_1(\mb{X})\gamma_2(\mb{Y})\expect_{\mb{X},\mb{Y}}[\he^q_k({X})\he^r_k({Y})]\right]} \nonumber \\
&\leq& \expect{\left[\mb{1_{\Delta_k}}(\mb{X},\mb{Y}) \gamma_1(\mb{X})\gamma_2(\mb{Y})\sqrt{\expect_{\mb{X}}[\he^{2q}_k({X})]\expect_{\mb{Y}}[\he^{2r}_k({Y})]}\right]} \nonumber \\
&=& \expect{\left[\mb{1_{\Delta_k}}(\mb{X},\mb{Y}) \gamma_1(\mb{X})\gamma_2(\mb{Y})\left({1_{\{q=1,r=1\}}}O\left(\frac{1}{k}\right) + o\left(\frac{1}{k}\right)\right)\right]} \nonumber \\
&=& \int{\left[\left({1_{\{q=1,r=1\}}}O\left(\frac{1}{k}\right) + o\left(\frac{1}{k}\right)\right)(\gamma_1(x)\gamma_2(x) + o(1))\right] \left(\int {\Delta_k}({x},{y}) dy \right)} dx  \nonumber \\
&=& \int{\left[\left({1_{\{q=1,r=1\}}}O\left(\frac{1}{k}\right) + o\left(\frac{1}{k}\right)\right)(\gamma_1(x)\gamma_2(x) + o(1))\right] \left(2^{d}\frac{k}{M} \right)} dx  \nonumber \\
&=& {1_{\{q=1,r=1\}}}c_{5,1}(\gamma_1,\gamma_2)\left(\frac{1}{M}\right) + o\left(\frac{1}{M}\right), \nonumber
\end{eqnarray}
where the bound is obtained using the Cauchy-Schwarz inequality and using Eq.\ref{fourmom}. Also,
\begin{eqnarray}
D &=&  \expect{\left[(1-\mb{1_{\Delta_k}}(\mb{X},\mb{Y})) \gamma_1(\mb{X})\gamma_2(\mb{Y}) \expect_{\mb{X},\mb{Y}}[Cov(\he^q_k({X}),\he^r_k({Y}))]\right]} \\ \nonumber 
&=&  {1_{\{q=1,r=1\}}}c_{5,2}(\gamma_1,\gamma_2)\left(\frac{1}{M}\right) + o\left(\frac{1}{M}\right). \nonumber
\end{eqnarray}
This gives 
\begin{equation}
Cov{\left[\gamma_1(\mb{X})\he^q_k(\mb{X}),\gamma_2(\mb{Y})\he^q_k(\mb{Y})\right]} = {1_{\{q=1,r=1\}}}c_5(\gamma_1(x),\gamma_2(x))\left(\frac{1}{M}\right) + o\left(\frac{1}{M}\right). \nonumber 
\end{equation}

Again, since $X \in {\cal S}_I(k)$ implies $\tilde{\mb{e}}_k(X) = \hat{\mb{e}}_k(X)$ and $Pr(\mb{X} \in {\cal S}_I(k)) = o(1)$,
\begin{equation}
Cov{\left[\gamma_1(\mb{X})\te^q_k(\mb{X}),\gamma_2(\mb{Y})\te^q_k(\mb{Y})\right]} = {1_{\{q=1,r=1\}}}c_5(\gamma_1(x),\gamma_2(x))\left(\frac{1}{M}\right) + o\left(\frac{1}{M}\right). \nonumber
\end{equation}

This concludes the proof.
\end{proof}

\section{Bias and variance results}
\label{sec:biasvarproof}
\newcommand{\ttt}{\check{\mb{f}}_k(\mb{Z})}
\newcommand{\ttti}{\check{{f}}_k(\mb{X}_i)}
\newcommand{\tttj}{\check{{f}}_k(\mb{X}_j)}
\newcommand{\tttt}{\check{{f}}_k(\mb{X}_2)}

\begin{lemma}
\label{boundonexpec}

Assume that $U(x,y)$ is any arbitrary functional which satisfies 
$$ (i) \sup_{y}|U(0,y)|  = G_1 < \infty, $$
$$ (ii) \sup_{x \in (p_l,p_u),y}|U(x,y)|  = G_2/4 < \infty, $$
$$(ii) \sup_{x \in (1/k,p_u),y }|U(x,y)|{\cal C}(M)  = G_3 < \infty,$$ 
$$(iii) \expect[\sup_{x \in (p_l,2^dM/k),y}|U(x,y)|]{\cal C}(M)   = G_4 < \infty.$$
Let $\mb{Z}$ denote $\mb{X}_i$ for some fixed $i \in \{1,..,N\}$. Let $\zeta_{\mb{Z}}$ be any random variable which almost surely lies in the range $(f(\mb{Z}),{\hat{\mb{f}}_k(\mb{Z})})$. Then, $$\expect[|U(\zeta_{\mb{Z}},{\mb{Z}})|] < \infty.$$
\end{lemma}

\begin{proof}

We will show that the conditional expectation $\expect[|U(\zeta_{Z},{Z})| \mid{\cal X}_N] < \infty.$ Because $0<\epsilon_0<f(X)<\epsilon_\infty<\infty$ by $({\cal {A}}.1)$, it immediately follows that $$ \expect[|U(\zeta_{\mb{Z}},{\mb{Z}})|] = \expect[\expect[|U(\zeta_{Z},{Z})| \mid {\cal X}_N]] < \infty.$$ Also observe that 
$\epsilon_0 < f(Z) < \epsilon_\infty$ and therefore $p_l < f(Z) <p_u$. Finally observe that the events $\natural_1(Z)$ and $\natural_2(Z)$ occur with probability $O({\cal C}(M))$. Using (\ref{densityineqnatural}), (\ref{densityineqnaturalc1}), (\ref{densityineqnaturalc2}), conditioned on ${\cal X}_N$,
\begin{eqnarray}
\expect[|U(\zeta_{Z},{Z})|] &=& \expect[1_{\natural_0(Z)}|U(\zeta_{Z},{Z})|] + \expect[1_{\natural_1(Z)}|U(\zeta_{Z},{Z})|] +  \expect[1_{\natural_2(Z)}|U(\zeta_{Z},{Z})|] + \expect[1_{\natural(Z)}|U(\zeta_{Z},{Z})|] \nonumber \\
&\leq& (G_1 + G_2) + (G_3+G_2) + (G_4 + G_2) + (G_2) \nonumber \\
&=& G_1+4G_2+G_3+G_4 < \infty.
\end{eqnarray}
\end{proof}

\textbf{Proof of Theorem~\ref{knnbiasH}}.
\begin{proof}

Using the continuity of $g'''(x,y)$, construct the following third order Taylor series of $g(\hat{\mb{f}}_k(\mb{Z}),\mb{Z})$ around the conditional expected value $\check{{f}}_k(\mb{Z}) = \expect[\hat{\mb{f}}_k(\mb{Z}) \mid \mb{Z}]$. 
\begin{eqnarray}
&&g({\hat{\mb{f}}_k(\mb{Z})},\mb{Z}) = g(\ttt,\mb{Z})+g'(\ttt,\mb{Z})\he_k(\mb{Z}) \nonumber \\
&& + \frac{1}{2}g''(\ttt,\mb{Z})\he_k^2(\mb{Z}) + \frac{1}{6}g^{(3)}(\zeta_\mb{Z},\mb{Z})\he_k^3(\mb{Z}),  \nonumber
\end{eqnarray}
where $\zeta_\mb{Z} \in (\ttt,{\hat{\mb{f}}_k(\mb{Z})})$ is defined by the mean value theorem. This gives
\begin{align}
 &\expect{[({g}(\hat{\mb{f}}_k(\mb{Z}),\mb{Z}) - {g}(\ttt,\mb{Z}))]} \nonumber \\
&= \expect{\left[\frac{1}{2}g''(\ttt,\mb{Z})\he_k^2(\mb{Z})\right]} + \expect{\left[\frac{1}{6}g^{(3)}(\zeta_\mb{Z},\mb{Z})\he_k^3(\mb{Z})\right]} \nonumber 
\end{align}
 Let $\Delta(\mb{Z}) = \frac{1}{6}g^{(3)}(\zeta_\mb{Z},\mb{Z})$. Direct application of Lemma~\ref{boundonexpec} in conjunction with assumption $({\cal {A}}.5)$  implies that $\expect[\Delta^2(\mb{Z})] = O(1)$. By Cauchy-Schwarz and applying Lemma~\ref{momentlemma} for the choice $q=6$, 
\begin{eqnarray}
&&\left| \expect{\left[\Delta(\mb{Z})\he_k^3(\mb{Z})\right]} \right| \leq \sqrt{\expect{\left[\Delta^2(\mb{Z})\right] \expect \left[\he_k^6(\mb{Z})\right]}} = o\left(\frac{1}{k}\right). \nonumber
\end{eqnarray}

By observing that the density estimates $\{\hat{\mb{f}}_k(\mb{X}_i)\}, i=1,\ldots,N$ are identical, we therefore have
\begin{eqnarray}
&&\expect[\hat{\mb{G}}_k] - G(f) = \expect{[{g}(\hat{\mb{f}}_k(\mb{Z}),\mb{Z}) - {g}({{f}(\mb{Z})},\mb{Z})]} \nonumber \\
&& = \expect{[{g}(\ttt,\mb{Z}) - {g}({{f}(\mb{Z})},\mb{Z})]} + \expect{\left[\frac{1}{2}g''(\ttt,\mb{Z})\mb{e}_k^2(\mb{Z})\right]} + o(1/k). \nonumber 
\end{eqnarray}
By Lemma~\ref{biaslemma} and Lemma~\ref{momentlemma} for the choice $q=2$, in conjunction with assumptions $({\cal {A}}.3)$ and $({\cal {A}}.4)$, this implies that
\begin{eqnarray}
\expect[\hat{\mb{G}}_k] - G(f) &=& \sum_{i=1}^d c_{1,i}(g(x,y))\left({\frac{k}{M}}\right)^{i/d}  +  c_2(g''(\check{f}_k(x),x))\left(\frac{1}{k}\right) + o\left(\frac{1}{k} + \frac{k}{M}\right) \nonumber \\
&=& \sum_{i=1}^d c_{1,i}(g(x,y))\left({\frac{k}{M}}\right)^{i/d}  + c_2(g''({f}(x),x))\left(\frac{1}{k}\right) +  o\left(\frac{1}{k} + \frac{k}{M}\right) \nonumber \\
&=& \sum_{i=1}^d c_{1,i}\left({\frac{k}{M}}\right)^{i/d} + c_2\left(\frac{1}{k}\right) + o\left(\frac{1}{k} + \frac{k}{M}\right), \nonumber
\end{eqnarray}
where the last but one step follows because, by (\ref{inbias}), we know $\check{f}_k({Z}) =  f({Z}) + o(1)$. This in turn implies $c_2(f^2(x)g''(\check{f}_k(x),x)) = c_2(f^2(x)g''({f}(x),x)) + o(1)$. Finally, by assumptions $({\cal {A}}.2)$ and $({\cal {A}}.4)$, the leading constants $c_{1,i}$ and $c_2$ are bounded. 

Note that the natural density estimate $\tilde{\mb{f}}_k(X)$ is identical to the truncated kernel density estimate $\hat{\mb{f}}_k(X)$ on the set $X \in {\cal S}_I(k)$. From the definition of set ${\cal S}_I(k)$, $Pr(\mb{Z} \notin {\cal S'}) = O((k/M)^{1/d}) = o(1)$.
\begin{eqnarray}
&&\expect[\tilde{\mb{G}}_k] - G(f) = \expect{[{g}(\tilde{\mb{f}}_k(\mb{Z}),\mb{Z}) - {g}({{f}(\mb{Z})},\mb{Z})]} \nonumber \\
&& = \expect{[1_{\{\mb{Z} \in {\cal S}_I(k)\}}{g}(\hat{\mb{f}}_k(\mb{Z}),\mb{Z}) - {g}({{f}(\mb{Z})},\mb{Z})]} + \expect{[1_{\{\mb{Z} \in {\cal S-S}_I(k)\}}{g}(\hat{\mb{f}}_k(\mb{Z}),\mb{Z}) - {g}({{f}(\mb{Z})},\mb{Z})]} \nonumber \\
&& = I+II
\end{eqnarray}

Using the exact same method as in the Proof of Theorem~\ref{knnbiasH}, using (\ref{inbias}) and (\ref{bknncent}), and the fact that $Pr(\mb{Z} \notin {\cal S}_I(k)) = O((k/M)^{1/d}) = o(1)$, we have
\begin{eqnarray}
I = c_{1,1}(g(x,y))\left({\frac{k}{M}}\right)^{1/d} + c_2(g''({f}(x)))\left(\frac{1}{k}\right) + o\left(\frac{1}{k} + \left(\frac{k}{M}\right)^{2/d}\right), \nonumber
\end{eqnarray}

Because we assume that $g$ satisfies assumption $({\cal {A}}.5)$, from the proof of Lemma~\ref{boundonexpec}, for ${Z} \in {\cal S-S}_I(k)$, we have $\expect{[{g}(\tilde{\mb{f}}_k({Z}),{Z}) - {g}({{f}({Z})},{Z})]} = O(1)$. This implies that, 
\begin{eqnarray}
II &=& \expect{[1_{\{\mb{Z} \in {\cal S-S}_I(k)\}}{g}(\hat{\mb{f}}_k(\mb{Z}),\mb{Z}) - {g}({{f}(\mb{Z})},\mb{Z})]} \nonumber \\
&=&\expect\left[\expect{[{g}(\hat{\mb{f}}_k({Z}),{Z}) - {g}({{f}({Z})},{Z})]}\mid {\{\mb{Z} \in {\cal S-S}_I(k)\}} \right] \times Pr(\mb{Z} \notin {\cal S}_I(k)) \nonumber \\
&=& O(1) \times O((k/M)^{1/d}) = O((k/M)^{1/d}).
\end{eqnarray}
This implies that 
\begin{eqnarray}
\expect[\tilde{\mb{G}}_k] - G(f) &=& I+II \nonumber \\
&=& c_{1}\left({\frac{k}{M}}\right)^{1/d} + c_2\left(\frac{1}{k}\right) + o\left(\frac{1}{k} + \left(\frac{k}{M}\right)^{1/d}\right). \nonumber
\end{eqnarray}

\end{proof}

\textbf{Proof of Theorem~\ref{knnvarH}}.
\begin{proof}
By the continuity of $g^{(\lambda)}(x,y)$, we can construct the following Taylor series of $g(\hat{\mb{f}}_k(\mb{Z}),\mb{Z})$ around the conditional expected value $\check{{f}}_k(\mb{Z})$.
\begin{eqnarray}
g(\hat{\mb{f}}_k(\mb{Z}),\mb{Z}) &=& g(\ttt,\mb{Z})+ {g'}(\ttt,\mb{Z})\he_k(\mb{Z}) \nonumber \\ 
&+& \left(\sum_{i=2}^{\lambda-1}\frac{g^{(i)}(\ttt,\mb{Z})}{i!}\he_k^i(\mb{Z})\right) + \frac{g^{(\lambda)}(\xi_\mb{Z},\mb{Z})}{\lambda!}\he_k^\lambda(\mb{Z}),  \nonumber
\end{eqnarray}
where $\xi_\mb{Z} \in (g(\expect_Z[\hat{\mb{f}}_k(\mb{Z})],g(\hat{\mb{f}}_k(\mb{Z})))$. Denote $(g^{\lambda}(\xi_\mb{Z},\mb{Z}))/\lambda!$ by $\Psi(\mb{Z})$. Further define the operator ${\cal M}(\mb{Z}) = \mb{Z} - \expect[\mb{Z}]$
and
\begin{eqnarray}
p_i &=& {\cal M}(g(\ttti,\mb{X_i})), \nonumber \\
q_i &=& {\cal M}({g'}(\ttti,\mb{X_i})\he_k(\mb{X_i})), \nonumber \\
r_i &=& {\cal M}\left(\sum_{i=2}^{\lambda}\frac{g^{(i)}(\ttti,\mb{X_i})}{i!}\he_k^i(\mb{X_i})\right) \nonumber \\
s_i &=& {\cal M}\left(\Psi(\mb{X_i})\he_k^{\lambda}(\mb{X_i})\right) \nonumber
\end{eqnarray}

The variance of the estimator $\hat{\mb{G}}_N(\mb{\hat{f}}_k)$ is given by
\begin{eqnarray}
&&\var[\hat{\mb{G}}_k] = \expect{[({\mb{\hat{G}}}(f)-\expect{[{\mb{\hat{G}}}(f)]})^2]} \nonumber \\
&& = \frac{1}{N}\expect{\left[(p{_1} + q{_1} + r{_1} + s_1)^2\right]} \nonumber \\
&& + \frac{N-1}{N}\expect{\left[(p{_1} + q{_1} + r{_1} + s_1)(p{_2} + q{_2} + r{_2} + s_2)\right]}. \nonumber
\end{eqnarray}
Because $\mb{X}_1$, $\mb{X}_2$ are independent, we have $\expect{\left[(p{_1})(p{_2} + q{_2} + r{_2} + s_2)\right]} = 0$. Furthermore,
\begin{eqnarray}
\expect{\left[(p{_1} + q{_1} + r{_1} + s_1)^2\right]} &=& \expect{[p{_1}^2]} + o(1) = \var[g(\check{{f}}_k(\mb{Z}),\mb{Z})] + o(1). \nonumber
\end{eqnarray}
Applying Lemma~\ref{momentlemma} and Lemma~\ref{covariancelemma}, in conjunction with assumptions $({\cal {A}}.3)$ and $({\cal {A}}.4)$, it follows that
\begin{itemize}
\item $\expect{[p{_1}^2]} =  \var[g(\ttt,\mb{Z})] = c_4(g(\check{f}_k(x),x))$
\item $\expect{\left[q{_1}q_{2}\right]} = c_5(g'(\check{f}_k(x),x),g'(\check{f}_k(x),x))\left(\frac{1}{M}\right) + o\left(\frac{1}{M}\right) $
\item $\expect{\left[q{_1}r_{2}\right]} =  o\left(\frac{1}{M}\right)$
\item $\expect{\left[r{_1}r_{2}\right]} = o\left(\frac{1}{M}\right) $
\end{itemize}

Since $q_1$ and $s_2$ are $0$ mean random variables
\begin{eqnarray}
&&\expect{\left[q_1s{_2}\right]} = \expect \left[q_1 \Psi(\mb{X}_2)(\hat{\mb{f}}_\mb{}(\mb{X}_2)-\tttt)^{\lambda} \right] \nonumber \\ 
&& = \expect \left[q_1 \Psi(\mb{X}_2)\he_k^\lambda(\mb{X}_2) \right] \nonumber \\ 
&& \leq  \sqrt{\expect \left[ \Psi^2(\mb{X_2})\right]\expect\left[q^2_1\he_k^{2\lambda}(\mb{X}_2) \right]} \nonumber \\
&& = \sqrt{\expect \left[ \Psi^2(\mb{Z})\right]}\left(o\left(\frac{1}{k^{\lambda}}\right) \right)\nonumber
\end{eqnarray}
Direct application of Lemma~\ref{boundonexpec} in conjunction with assumptions $({\cal {A}}.5)$ implies that $\expect \left[\Psi^2(\mb{Z})\right] = O(1)$. Note that from assumption $({\cal {A}}.3)$, $o\left(\frac{1}{k^{\lambda}}\right) = o(1/M)$ . In a similar manner, it can be shown that $\expect{\left[r{_1}s_{2}\right]} = o\left(\frac{1}{M}\right)$ and $\expect{\left[s{_1}s_{2}\right]} = o\left(\frac{1}{M}\right) $. This implies that 
\begin{eqnarray}
\var[\hat{\mb{G}}_k] &=& \frac{1}{N}\expect{\left[p{_1}^2\right]} + \frac{(N-1)}{N}\expect{\left[q{_1}q_{2}\right]} \nonumber +  o\left(\frac{1}{M}+\frac{1}{N}\right) \nonumber \\
&=& c_4(g(\check{f}_k(x),x))\left(\frac{1}{N}\right)+ c_5(g'(\check{f}_k(x),x),g'(\check{f}_k(x),x))\left(\frac{1}{M}\right)  + o\left(\frac{1}{M} + \frac{1}{N}\right) \nonumber \\
&=& c_4(g({f}(x),x))\left(\frac{1}{N}\right)+ c_5(g'({f}(x),x),g'({f}(x),x))\left(\frac{1}{M}\right) + o\left(\frac{1}{M} + \frac{1}{N}\right) \nonumber \\
&=& c_4\left(\frac{1}{N}\right)+ c_5\left(\frac{1}{M}\right)  + o\left(\frac{1}{M} + \frac{1}{N}\right), \nonumber
\end{eqnarray}
where the last but one step follows because, by (\ref{inbias}), we know $\check{f}_k({Z}) =  f({Z}) + o(1)$. This in turn implies $c_4(g(\check{f}_k(x),x)) = c_4(g({f}(x),x))+o(1)$ and $ c_5(g'(\check{f}_k(x),x),g'(\check{f}_k(x),x)) = c_5(g'({f}(x),x),g'({f}(x),x)) + o(1)$. Finally, by assumptions $({\cal {A}}.2)$ and $({\cal {A}}.4)$, the leading constants $c_4$ and $c_5$ are bounded. 

Because of the identical nature of the expressions of $\hat{\mb{e}}_k(X)$ and $\tilde{\mb{e}}_k(X)$ in Lemma~\ref{momentlemma} and Lemma~\ref{covariancelemma}, it immediately follows that
\begin{eqnarray}
\var[\tilde{\mb{G}}_k] &=& c_4\left(\frac{1}{N}\right)+ c_5\left(\frac{1}{M}\right)  + o\left(\frac{1}{M} + \frac{1}{N}\right). \nonumber
\end{eqnarray}

This concludes the proof of Theorem~\ref{knnvarH}.

\end{proof}

\bibliographystyle{plain}
\footnotesize{\bibliography{EnsembleEstimators.bib}}

\begin{thebibliography}{10}

\bibitem{beir}
J.~Beirlant, EJ~Dudewicz, L.~Gy{\"o}rfi, and EC~Van~der Meulen.
\newblock Nonparametric entropy estimation: An overview.
\newblock {\em Intl. Journal of Mathematical and Statistical Sciences},
  6:17--40, 1997.

\bibitem{birge}
L.~Birge and P.~Massart.
\newblock Estimation of integral functions of a density.
\newblock {\em The Annals of Statistics}, 23(1):11--29, 1995.

\bibitem{heroJ}
J.A. Costa and A.O. Hero.
\newblock Geodesic entropic graphs for dimension and entropy estimation in
  manifold learning.
\newblock {\em Signal Processing, IEEE Transactions on}, 52(8):2210--2221,
  2004.

\bibitem{fuk2}
K.~Fukunaga and L.~D. Hostetler.
\newblock Optimization of k-nearest-neighbor density estimates.
\newblock {\em IEEE Transactions on Information Theory}, 1973.

\bibitem{gini}
E.~Gin\'e and D.M. Mason.
\newblock Uniform in bandwidth estimation of integral functionals of the
  density function.
\newblock {\em Scandinavian Journal of Statistics}, 35:739�761, 2008.

\bibitem{riten}
R.~Gupta.
\newblock {\em Quantization Strategies for Low-Power Communications}.
\newblock PhD thesis, University of Michigan, Ann Arbor, 2001.

\bibitem{gyrofi0}
L.~Gy{\"{o}}rfi and E.~C. van~der Meulen.
\newblock Density-free convergence properties of various estimators of entropy.
\newblock {\em Comput. Statist. Data Anal.}, pages 425--436, 1987.

\bibitem{gyrofi}
L.~Gy{\"{o}}rfi and E.~C. van~der Meulen.
\newblock An entropy estimate based on a kernel density estimation.
\newblock {\em Limit Theorems in Probability and Statistics}, pages 229--240,
  1989.

\bibitem{hero}
A.~O. Hero, J.~Costa, and B.~Ma.
\newblock Asymptotic relations between minimal graphs and alpha-entropy.
\newblock {\em Technical Report, Communications and Signal Processing
  Laboratory, The University of Michigan}, March 2003.

\bibitem{multipleK}
G.~Lanckriet, N.~Cristianini, P.~Bartlett, and L.~El Ghaoui.
\newblock Learning the kernel matrix with semi-definite programming.
\newblock {\em Journal of Machine Learning Research}, 5:2004, 2002.

\bibitem{laurent}
B.~Laurent.
\newblock Efficient estimation of integral functionals of a density.
\newblock {\em The Annals of Statistics}, 24(2):659--681, 1996.

\bibitem{leo2}
N.~Leonenko, L.~Prozanto, and V.~Savani.
\newblock A class of {R\'enyi} information estimators for multidimensional
  densities.
\newblock {\em Annals of Statistics}, 36:2153--2182, 2008.

\bibitem{litt}
E.~Liiti\"{a}inen, A.~Lendasse, and F.~Corona.
\newblock On the statistical estimation of r\'{e}nyi entropies.
\newblock In {\em Proceedings of {IEEE}/{MLSP} 2009 International Workshop on
  Machine Learning for Signal Processing, Grenoble (France)}, September 2-4
  2009.

\bibitem{pal}
D.~Pal, B.~Poczos, and C.~Szepesvari.
\newblock Estimation of {R\'enyi} entropy and mutual information based on
  generalized nearest-neighbor graphs.
\newblock In {\em Proc. Advances in Neural Information Processing Systems
  (NIPS)}. MIT Press, 2010.

\bibitem{raykar}
V.~C. Raykar and R.~Duraiswami.
\newblock Fast optimal bandwidth selection for kernel density estimation.
\newblock In J.~Ghosh, D.~Lambert, D.~Skillicorn, and J.~Srivastava, editors,
  {\em Proceedings of the sixth SIAM International Conference on Data Mining},
  pages 524--528, 2006.

\bibitem{boosting}
Robert~E. Schapire.
\newblock {The strength of weak learnability}.
\newblock {\em Machine Learning}, 5(2):197--227--227, June 1990.

\bibitem{kks}
K.~Sricharan, R.~Raich, and A.~O. Hero.
\newblock {Empirical estimation of entropy functionals with confidence}.
\newblock {\em ArXiv e-prints}, December 2010.

\bibitem{kde}
B.~Turlach.
\newblock Bandwidth selection in kernel density estimation: A review.

\bibitem{wang}
Q.~Wang, S.~R. Kulkarni, and S.~Verd{\'u}.
\newblock {Divergence estimation of continuous distributions based on
  data-dependent partitions}.
\newblock {\em Information Theory, IEEE Transactions on}, 51(9):3064--3074,
  2005.

\end{thebibliography}

\end{document}